\documentclass[12pt,a4paper]{article}
\usepackage[latin1]{inputenc}
\usepackage{amsmath}
\usepackage{amsfonts}
\usepackage{amssymb}
\usepackage{amsthm}
\usepackage[english]{babel}
\usepackage{amsmath}
\usepackage{amsfonts}
\usepackage{amssymb}
\usepackage{amstext}
\usepackage{makeidx}
\usepackage{graphicx}

\theoremstyle{plain}
\usepackage{xcolor}
\usepackage{geometry}
\newtheorem{theorem}{Theorem}[section]
\newtheorem{lemma}[theorem]{Lemma}
\newtheorem{proposition}[theorem]{Proposition}

\newtheorem{remark}[theorem]{Remark}

\newtheorem{corollary}[theorem]{Corollary}
\newtheorem{question}[theorem]{Question}
\newtheorem*{theorem*}{Theorem}
\font\eusmtwelve=eusm10 at 12pt
\font\eusmeight=eusm7 at 8pt
\font\eusmsix=eusm5 at 6pt

\newcommand{\mathscript}[1]
{\mathchoice
{\text{\eusmtwelve #1}}
{\text{\eusmtwelve #1}}
{\text{\eusmeight #1}}
{\text{\eusmsix #1}}}

\newcommand{\A}{\mathscript A}
\newcommand{\B}{\mathscript B}
\newcommand{\C}{\mathscript C}
\newcommand{\D}{\mathscript D}
\newcommand{\N}{\mathscript N}
\newcommand{\Q}{\mathscript Q}
\renewcommand{\S}{\mathscript S}
\newcommand{\T}{\mathscript T}

\newcommand{\Cl}[1]{\mathop{\mathrm{Cl}}#1}
\newcommand{\Fix}[1]{\mathop{\mathrm{Fix}}#1}

\newcommand{\id}{\mathrm{id}}
\newcommand{\st}{\;|\;}

\newcommand{\cyl}[1]{[\mkern0.5mu#1\mkern0.5mu]}

\newcommand{\per}[1]{#1^\infty}
\newcommand{\pper}[1]{(#1)^\infty}

\newcommand{\restr}[1]{|_{#1}}
\newcommand{\trunc}[1]{\smash{\text{{\raise 0.05ex \hbox{$\shortmid$}}%
                              \llap{\raise -0.55ex \hbox{$\shortmid$}}}}_{#1}}
                               
\newcommand{\llangle}{\langle\mkern-4.5mu\langle}
\newcommand{\rrangle}{\rangle\mkern-4.5mu\rangle}

\newcommand{\binexp}[2]{{\langle{#1}\rangle}^{\mkern-3mu{#2}}}
\newcommand{\fixexp}[1]{{\langle{#1}\rangle}^{\mkern-3mu\scriptscriptstyle\text{F}}}

\newcommand{\fibexp}[1]{\llangle{#1}\rrangle}

\newcommand\hiddentext[1]{}

\begin{document}

\parindent0pt
\frenchspacing

\let\setminus\smallsetminus


\thispagestyle{empty}

\vglue9mm

\centerline{\LARGE\bf The simplest erasing substitution}

\vskip 6mm

\centerline{\large\it\hfil{Alessandro Della Corte}\hfil{Stefano Isola}\hfil{Riccardo Piergallini}\hfil}

\vskip3mm

\centerline{\large{Scuola di Scienze e Tecnologie -- Universit\`a di Camerino -- Italy}}

\vskip6mm

\begin{abstract}

\noindent
In this work, we begin the study of a new class of dynamical systems determined by interval maps generated by the symbolic action of erasing substitution rules. We do this by discussing in some detail the geometric, analytical, dynamical and arithmetic properties of a particular example, which has the virtue of being arguably the simplest and that at the same time produces interesting properties and new challenging problems.


\end{abstract}


\section*{Introduction}

Substitutive dynamical systems are a widely studied and quite well-understood class (see for instance \cite{Fogg,Queffe}). In this context, it is usual to intend the term \emph{substitution} in a rather specific sense. Namely, a substitution $S$ is typically understood as a rule replacing every element from a given alphabet by a finite word on the same alphabet, in order to extend $S$ by concatenation to a morphism over all the finite or infinite words. The morphic nature of $S$ means that the action of the substitution on a certain symbol within a word $w$ is independent of its position in $w$. Moreover, most of the times no symbol is mapped by $S$ to the empty word, that is the substitution is assumed to be \emph{non-erasing}.

\medskip

In recent years, erasing substitutions have been taken in some consideration within combinatorics of words (see for instance \cite{Durand,Schneider}), mostly in the context of computer science, while little attention seems to be devoted to the analytical and dynamical properties of real maps generated by the symbolic action of erasing substitutions. The shift from symbolic spaces to the continuous real context, that is from zero-dimensional spaces to one-dimensional ones, has relevant consequences on the richness of the properties involved and the problems that seems natural to consider. Lately, topological dynamics questions initially considered only in case of continuous maps, such as the structure of $\omega$-limit sets, topological entropy, the presence of Devaney, Li-Yorke or distributional chaos and mixing properties, have been addressed in a more general context. For instance: Darboux interval maps of Baire class 1 are considered in \cite{KK,Szuca}, general Darboux interval maps are investigated from the point of view of the connections between transitivity, turbulence and topological entropy in \cite{Pawlak,Pawlak2}, interval maps with $G_\delta$ connected graph are studied in \cite{Ciklova1,Ciklova2}, while in \cite{Kahng} an adjustment of the concept of Devaney chaos is proposed for discontinuous maps. In this connection, the study of the dynamics of real maps generated by erasing block substitutions appears as a natural development, as the objects constructed in this way fall into a category which is in a way the direct generalization of the ones cited above, that is Baire class 1, generally not Darboux functions. The systematic investigation of the dynamical properties of these maps has just begun; for instance,  generic interval maps of Baire class 1 are studied in \cite{Steele0,Steele1,Hanson}.

\medskip

Our aim is the investigation of an interval map generated by what is arguably the simplest erasing substitution rule, which we will indicate by $\rho$. We are led to the choice of $\rho$ as follows. We start selecting the smallest non-trivial alphabet, that is $\{0,1\}$. The simplest way to get an erasing substitution is to map one of the symbols, say 0, to the empty word. Then, we must make a choice on how to transform the other symbol 1. Clearly, to obtain a non-trivial map, we cannot replace every 1 with a fixed word, independently of its position. Making once again the simplest choice, we distinguish between odd and even positions, replacing one case with the symbol 0 and the other with the symbol 1. 

\medskip

Even in this model-case, the real map $R:[0,1] \to [0,1]$ generated by the symbolic action of $\rho$ on binary expansions, presents interesting properties and challenging problems, which involve, among others, measure-theoretic, dimension-theoretic, topological, dynamical and arithmetical aspects. As we will see better in Section \ref{map}, from an operational point of view the map is hardly manageable, as to obtain the first $m$ digits of the $n$-th iterate of its action at point $x$ we must typically provide x with accuracy of order $\approx 2^{-2^{mn}}$. This entails that it is almost impossible to investigate in a purely numerical fashion objects as $\omega$-limit sets and attractors. Moreover, as we will see, there is an uncountable and dense subset of $[0,1]$ exhibiting extreme distributional chaos, which means that the qualitative relative behaviour of sets of points, even arbitrarily close, cannot be extrapolated from finite samplings of the dynamics of any time length.

\medskip

It is this combination of simplicity of the object and variety of related properties which in our opinion confers interest to the proposed subject. 

\medskip

The main properties of the map $R$, defined in Sections \ref{map} and \ref{function}, can be summarized as follows.

\vskip-7.5pt\vskip0pt\leftmargini25pt
\begin{enumerate}\topsep0pt\itemsep0pt

\item The graph of $R$ is a totally disconnected subset of the plane, whose Hausdorff dimension is between 1 and $\log_2{3}$ (Section \ref{graph}).

\item $R$ is a singular Borel map of Baire class 1, it is not Darboux, and its mean value is the (admittedly a bit surprising) rational number $3/7$ (Section \ref{analysis}).

\item The closure of the fibers and of the fixed point set of $R$ are null-measure Cantor sets, with Hausdorff dimension belonging to $[1/2, \log_2{\varphi}]$ (Sections \ref{analysis} and \ref{dynamics}).

\item There are two rational $R$-cycles of order two which attract (in a finite time) every rational, moreover $R$ has uncountably many periodic points for any given period, it exhibits Devaney and (uniform) DC1 chaos, it is topologically mixing, it has infinite topological entropy and every point in $[0,1]$ is a full entropy point (Section \ref{dynamics}).

\item The combinatorially simplest periodic points of odd period (including 1) are transcendental numbers and there are uncountably many periodic points that are transcendental as well for every given period (Section \ref{arithmetic}).

\end{enumerate}
\vskip-6pt

For the reader's convenience, a list of the most used symbols is given at the end of the paper.


\section{The general setting}
\label{setting}

Let $\A = \mathbb N^\omega$ be the set of all infinite sequences of non-negative integers, and $\B = \{0,1\}^\omega \subset \mathbb N^\omega$ be the subset of all infinite binary sequences, both endowed with the metric given by
\begin{equation*}
d(s,t) = 1/2^{\min\{k\geq 1 \st s_k\neq\,t_k\}},
\end{equation*}
for every $s= (s_1,s_2,\dots,s_n,\dots)$ and $t= (t_1, t_2, \dots, t_n,\dots)$ such that $s \neq t$.

\medskip

Then, $\B$ splits as the union of two dense subspaces
\begin{equation*}
\B = \C \cup \C',
\end{equation*} 
where $\C$ consists of all the infinite binary sequences which are not eventually 0, while $\C'$ consists of all the infinite binary sequences which are not eventually 1.

\medskip

There is a 1-Lipschitz homeomorphism $\eta: \A \to \C$ defined by
\begin{equation}\label{etadef}
\begin{array}{rcl}
a = (a_1,a_2,\dots,a_n,\dots\,) \ \mathrel{\buildrel \textstyle \eta \over \longmapsto}\ \eta(a) 
\!\!\!&=&\!\!\! (0^{a_1},1,0^{a_2},1,\dots,0^{a_n}1,\dots\,)\\[4pt]
&=&\!\!\! (\underbrace{0,\dots,0}_{a_1},1,\underbrace{0,\dots,0}_{a_2},1,\dots, \underbrace{0,\dots,0}_{a_n},1,\dots\,) \in \C
\end{array}
\end{equation}
for every $a = (a_1,a_2,\dots,a_n,\dots\,) \in \A$.

\medskip

For the sake of convenience, we introduce the notation
\begin{equation}\label{binexp}
\binexp{a}{b} = \eta(2a + \eta^{-1}(b))
\end{equation}
with $a \in \A$ and $b \in \C$. If $a = (a_1,a_2,\dots,a_n,\dots\,)$ then $\binexp{a}{b}$ is obtained from $b$ by inserting $0^{2a_k}$ immediately before the $k$-th occurrence of 1 in $b$. In particular, $\binexp{a}{\per{1}} = \binexp{a}{(1,1,\dots)} = \eta(2a)$.

\medskip

There is also a 2-Lipschitz map $\xi:\B \to [0,1]$ defined by
\begin{equation}\label{xidef}
b = (b_1,b_2,\dots,b_n,\dots\,) \ \mathrel{\buildrel \textstyle\xi \over\longmapsto}\ \xi(b) = \sum_{n = 1}^\infty \frac{b_n}{2^n} = 0.b_1b_2\dots b_n \ldots \in [0,1]
\end{equation} 
for every $b = (b_1,b_2,\dots,b_n,\dots\,) \in \B$.

\medskip

The map $\xi$ splits as the union of two bijective maps
\begin{equation*}
\xi\restr{\C}:\C \to (0,1] \ \text{ and } \ \xi\restr{\C'}:\C' \to [0,1).
\end{equation*}

Then, we have the inverse maps
\begin{equation}\label{zetadef}
\beta = (\xi\restr{\C})^{-1}: (0,1] \to \C \ \text{ and } \ 
\beta' = (\xi\restr{\C'})^{-1}:[0,1) \to \C',
\end{equation}
giving the binary expansions of the real numbers in $[0,1]$, meaning the infinite sequences of binary digits in their binary expressions, completed by infinitely many 0's in the case of the finite binary expressions of dyadic rationals.

\medskip

We emphasize the difference between a \emph{binary expression} $0.x_1x_2\dots$ of $x \in [0,1]$, which can be finite (for dyadic rationals) or infinite, and a \emph{binary expansion} of $x$ which is always an infinite sequence in $\B$. Namely, for a dyadic rational $x \in (0,1]$ admitting the finite binary expression $0.x_1x_2\dots x_n$, we have $\beta'(x) = (x_1,x_2,\dots,x_n,0,0,\dots) \in \C'$, while we have $\beta(x) = (x_1,x_2,\dots,x_{k-1},0,1,1,\dots) \in \C $ if $k \leq n$ is the index of the last occurrence of 1 in $0.x_1x_2\dots x_n$.

\medskip

Now, let $\mathbb D \subset \mathbb Q$ denote the set of all dyadic rationals, and put
\begin{equation*}
\D = \mathbb D \cap [0,1] \ \text{ and } \ \Q = \mathbb Q \cap [0,1]\,.
\end{equation*}
Notice that the two maps $\beta$ and $\beta'$ coincide and are continuous on $(0,1) \setminus\D$, while they differ on the set $\D \setminus \{0,1\}$, where $\beta$ is only left-continuous and $\beta'$ is only right-continuous.

Finally, we introduce the further notations $b\trunc{k}$ for the $k$-th truncation of an infinite binary sequence $b$, and $\cyl{w}$ for the cylinder consisting of all the infinite binary sequences having the finite prefix $w$. Namely,
\begin{eqnarray}
\label{trunc}
& b\trunc{k} = (b_1,b_2,\dots,b_k) \in \{0,1\}^k\ \text{ for every $b = (b_1,b_2,\dots,b_n,\dots) \in \B\,$,}\\[3pt]
\label{cyl}
& \cyl{w} = \{(w,b) \st b \in \B\} \subset \B \ \text{ for every $w = (w_1,w_2,\dots,w_k) \in \{0,1\}^k.$}
\end{eqnarray}
Then, for every $b \in \B$ and $k \geq 0$ we can express the open ball $B(b,1/2^k) \subset \B$ as $\cyl{b\trunc{k}}\,$.

\medskip

In the following, we identify the infinite binary sequence $b = (b_1,b_2, \dots)$ with the infinite binary word $b = b_1b_2\dots$, and we write $x = 0.b$ for (the corresponding binary expression of) $x = \xi(b)$. An analogous convention is adopted for finite binary sequences. Thus, the notations introduced in the equations \eqref{binexp}, \eqref{trunc} and \eqref{cyl} make sense also in the context of binary words, and hence in the context of binary expressions/expansions of real numbers in $[0,1]$.


\section{The substitution rule}
\label{map}

Let $\{0,1\}^\infty$ be the set of all the finite or infinite binary words, including the empty word $\epsilon$.\break For a word $w \in \{0,1\}^\infty$, we indicate by $|w|$ its (possibly infinite) length and by $|w|_1$ the (possibly infinite) number of 1's occurring in it.

\medskip

Consider the map $\rho:\{0,1\}^\infty \to \{0,1\}^\infty$ digit-wise defined by the alternating substitution rule
\begin{equation}\label{rho}
\rho:\left\{\vrule width0pt height22pt\right.\!\!\!
\begin{array}{ll}
0 \mapsto \epsilon\\
1 \mapsto 0 & \text{in odd positions}\\
1 \mapsto 1 & \text{in even positions}
\end{array}.
\end{equation}
In other words, for any word $b = b_1b_2 \ldots \in \{0,1\}^\infty$ the word $\rho(b)$ is obtained from $b$ by deleting all the $b_n = 0$ and replacing each $b_n = 1$ by $0$ if $n$ is odd and by $1$ if $n$ is even.

\medskip

We remark that the parity condition in the above definition makes $\rho$ a non-morphic map with respect to the concatenation. Indeed, for any $v,w \in \{0,1\}^\infty$ with $v$ a finite word and $|w|_1 > 0$, the equality
$\rho(vw) = \rho(v)\rho(w)$ holds only if $|v|$ is even, while for $|v|$ odd we have $\rho(vw) = \rho(v)\widetilde\rho(w)$, where $\widetilde\rho$ is the complementary substitution rule
\begin{equation}\label{tilderho}
\widetilde\rho:\left\{\vrule width0pt height22pt\right.\!\!\!
\begin{array}{ll}
0 \mapsto \epsilon\\
1 \mapsto 1 & \text{in odd positions}\\
1 \mapsto 0 & \text{in even positions}
\end{array}.
\end{equation}
For the sake of convenience, we introduce the notation
\begin{equation}\label{rhok}
\rho_k = 
\left\{\vrule width0pt height12pt\right.\!\!\!
\begin{array}{ll}
\rho & \text{if $k$ is an even integer}\\
\widetilde\rho & \text{if $k$ is an odd integer}
\end{array},
\end{equation}
and we write $\rho_v$ to mean $\rho_{|v|}$. In this way, for any $v$ and $w$ as above we have
\begin{equation}\label{rhov}
\rho(vw) = \rho(v)\rho_v(w).
\end{equation}

\medskip

We observe that, for binary words of even or infinite length, the action of $\rho$ coincides with that of the block substitution rule
\begin{equation}\label{tau}
\tau:\left\{\vrule width0pt height28pt\right.\!\!\!
\begin{array}{ll}
00 \mapsto \epsilon\\
01 \mapsto 1\\
10 \mapsto 0\\
11 \mapsto 01
\end{array}.
\end{equation}

\medskip

This interpretation of the substitution rule $\rho$ enables us to estimate the \emph{vanishing order}
\begin{equation}\label{vanord}
n_\epsilon(w) = \min\{k > 0 \st \rho^k(w)=\epsilon\}
\end{equation}
of a finite binary word $w$ under the action of $\rho$ in terms of $|w|$.

\begin{proposition}\label{erasing}
If $w \neq \epsilon$ is a finite binary word then $n_\epsilon(w) \leq 2\lfloor \log_2 |w|\rfloor + 2$. 
\end{proposition}

\begin{proof}
We proceed by induction on $|w| \geq 1$, based on the trivial case of $|w| = 1$. So, let $|w| > 1$ and
observe that we can assume $|w|$ even, because of the obvious equality $n_\epsilon(w) = n_\epsilon(w0)$. In this case, we have $|\rho^2(w)| = |\rho(\tau(w))| = |\tau(w)|_1$. Looking at the contribution to $|\tau(w)|_1$ of each single pair, according to \eqref{tau}, we immediately get $|\rho^2(w)| \leq |w|/2$. Then, by the inductive hypothesis,
\begin{equation*}
n_\epsilon(w) \leq n_\epsilon(\rho^2(w)) + 2 \leq 2\lfloor \log_2 |\rho^2(w)|\rfloor + 4 \leq 
2\lfloor \log_2 (|w|/2)\rfloor + 4 = 2\lfloor \log_2 |w|\rfloor + 2.
\end{equation*}
\vskip-\lastskip
\vskip-\baselineskip
\end{proof}
\vskip12pt

As it can be easily realized, the vanishing order $n_\epsilon(w)$ of a word of length $n \geq 1$ can assume any value between 1 and the upper bound established by the proposition. In particular, the extremal values are attained for the words $0^n$ and $1^n$, which vanish exactly after 1 and $2\lfloor\log_2n\rfloor + 2$ iterations of $\rho$, respectively.

\medskip

Up to the identification between binary sequences and binary words, we have $\C \subset \B \subset \{0,1\}^\infty$. Moreover, it can be easily seen that $\rho^{-1}(\B) = \C$. Then, it makes sense to consider the restriction $\rho\restr{\C} = \tau\restr{\C}: \C \to \B$.
\medskip

Now, consider the substitution rule
\begin{equation}\label{tauinv}
\overline\tau:\left\{\vrule width0pt height12pt\right.\!\!\!
\begin{array}{ll}
0 \mapsto 10\\
1 \mapsto 01
\end{array},
\end{equation}
and let $\overline\tau:\{0,1\}^\infty \to \{0,1\}^\infty$ be the corresponding digit-wise generated map. Then, the image $\overline\tau(w)$ has even or infinite length for every $w \in \{0,1\}^\infty$. Moreover, $\tau \circ \overline\tau = \id_{\{0,1\}^\infty}$ and $\overline\tau(w) \in \C$ for every $w \in \B$. It immediately follows that the map $\rho\restr{\C}$ is surjective.

\medskip

Actually, $\overline\tau(w)$ is not the ``simplest'' element in $(\rho\restr{\C})^{-1}(w)$ for $w \in \B$. In fact, $\overline\tau(w)$ can contain pairs of consecutive 0's and these can be deleted without changing the image under $\rho\restr{\C}$. The reason is that $\rho(0) = \epsilon$ and that deleting/inserting subwords of even length does not change the parity of the following positions in the word.

\medskip

In the following we indicate by $\S \subset \C$ the set of the ``simplest'' infinite binary words, meaning those which do not contain any pair of consecutive 0's.

\medskip

In order to directly construct the ``simplest'' element in $(\rho\restr{\C})^{-1}(w)$, the unique one which belongs to $\S$, we consider the section $\sigma:\{0,1\}^\infty \to \{0,1\}^\infty$ digit-wise generated by the rule
\begin{equation}\label{sigma}
\sigma:\left\{\vrule width0pt height12pt\right.\!\!\!
\begin{array}{ll}
x \mapsto 1 & \text{if $x$ is preceded by $1-x$ in the word or it is 0 as the first digit}\\
x \mapsto 01 & \text{if $x$ is preceded by $x$ in the word or it is 1 as the first digit}
\end{array}.
\end{equation}
Given $w \in \{0,1\}^{\infty}$, $\sigma(w)$ is actually given by a true substitution rule (in the usual sense) applied to the sequence of first differences of the word $1w$ obtained prepending 1 to $w$. Moreover, $\sigma(b) \in \S \subset \C$ for every $b \in \B$, and hence $\sigma\restr{\B}: \B \to \C$ is a section of $\rho\restr{\C}$.

\begin{proposition}\label{fiber}
For every $b \in \B$ the fiber $(\rho\restr{\C})^{-1}(b)$ contains uncountably many elements.\break In fact,
\begin{equation}\label{fibereq}
(\rho\restr{\C})^{-1}(b) = \{\binexp{a}{\sigma(b)} \st a \in \A\},
\end{equation}
and hence $\sigma(b)$ is the unique element of $(\rho\restr{\C})^{-1}(b)$ belonging to $\S$ and it is the maximum of $(\rho\restr{\C})^{-1}(b)$ with respect to the lexicographic order.
\end{proposition}

\begin{proof}
We only need to prove equation \eqref{fibereq}, since the rest of the statement immediately follows from it. A straightforward induction on $n$ gives $\rho\restr{\C}(\sigma(b))_n = b_n$ for every $n \geq 1$, which implies that $\rho(\sigma(b)) = b$. On the other hand, by deleting pairs of 0's as discussed above, we get $\rho(\binexp{a}{\sigma(b)}) = b$ for every $a \in \A$. Conversely, let any $v \in (\rho\restr{\C})^{-1}(b)$ be written as $v = \eta(c)$, thanks to \eqref{etadef}. The $n$-th occurrence of 1 in $v$ is in position $k_n = c_1 + \ldots + c_n + n$. Since the parity of $k_n$ is determined by $\rho(v)_n = b_n$, induction on $n$ shows that also the parity of $c_n$ is determined for every $n \geq 1$. So, if $\sigma(b) = \eta(d)$ then we have $d_n = 0,1$ and $c_n \equiv d_n \text{ mod } 2$ for every $n\geq 1$. This yields $c = 2a + d$ for a suitable $a \in \A$, and hence, according to \eqref{binexp}, $v = \eta(c) = \eta(2a + d) = \binexp{a}{\sigma(b)}$.
\end{proof}

We conclude this section by considering some continuity properties of the maps $\rho\restr{\C}: \C \to \B$ and $\sigma\restr{\B}: \B \to \C$ with respect to the metric induced by the inclusions $\C \subset \B \subset \A$.

\begin{proposition}\label{rhocont}
The map $\rho\restr{\C}: \C \to \B$ is continuous but not uniformly continuous. Moreover, $\lim_{v \to b} \rho(v)$ does not exists for any $b \in \B \setminus\C$. So, $\rho$ cannot be continuously extended to any subset of $\,\B$ larger than $\C$.
\end{proposition}

\begin{proof}
For any $b \in \C$ and $n \geq 1$, let $k_n = |\rho(b\trunc{n})| = |b\trunc{n}|_1$. Then, $\rho(B(b,1/2^n)) = \rho(\cyl{b\trunc{n}}) \subset \cyl{\rho(b\trunc{n})} = \cyl{\rho(b)\trunc{k_n})} = B(\rho(b),1/2^{k_n})$, and hence the continuity of $\rho\restr{\C}$ at $b$ immediately follows from the fact that $k_n \to \infty$ for $n \to \infty$. On the other hand, if $\rho$ were uniformly continuous then it would be possible to extend it to the whole space $\{0,1\}^\infty$, and this would contradict the second part of the statement, which we are going to prove.

Let $b \in \B \setminus \C$. Then, $b$ is eventually 0, and hence it can be written as $b = v \per{0}$. Consider the sequence $(b_n = v0^n\per{1})_{n \geq 1} \subset \C$. We have, $\lim_{n \to \infty} b_n = b$, while the limit $\lim_{n \to \infty} \rho(b_n)$ cannot exist, being $\rho(b_{2k+1}) = \rho(b_1) \neq \rho(b_2) = \rho(b_{2k+2})$ for every $k \geq 1$.
\end{proof}

\begin{proposition}\label{sigmacont}
The map $\sigma\restr{\B}: \B \to \C$ is uniformly continuous, in fact it is $1$-Lipschitz.
\end{proposition}

\begin{proof}
Since $|\sigma(v)| \geq k$ for every $v \in \{0,1\}^k$, we have that
$d(b,b') = 1/2^k \Rightarrow b\trunc{k} = b'\trunc{k} \Rightarrow \sigma(b)\trunc{k} =
\sigma(b')\trunc{k} \Rightarrow d(\sigma(b),\sigma(b')) \leq 1/2^k$, for every $b,b' \in \B$.
\end{proof}


\section{The real map}
\label{function}

We define the map $R:[0,1] \to [0,1]$ as follows
\begin{equation}\label{R}
R(x) = \left\{\vrule width0pt height14pt\right.\!\!\!
\begin{array}{ll}
2/3 & \text{if $x = 0$}\\
\xi(\rho(\beta(x))) & \text{if $x \in (0,1]$}
\end{array}.
\end{equation}
For $x \neq 0$ this amounts to say that $R(x)$ admits a binary expansion which is the transform under $\rho$ of the unique binary expansion of $x$ in $\C$. In formulas, if $x = 0.b$ with $b \in \C$ then
\begin{equation*}
R(x)=R(0.b)=0.\rho(b)\,.
\end{equation*}
The image of 0 could be chosen somewhat arbitrarily. This particular choice is convenient for later purposes.

\medskip

The continuity properties we have seen in the previous section for the maps involved in the definition of $R$ give us the next proposition.

\begin{proposition}\label{Rcont}
The map $R$ is continuous everywhere but on the set $\D \setminus \{1\}$, where it is only left-continuous. Hence, $R$ is a Borel function of Baire class $1$, i.e.\ it is a point-wise limit of continuous functions.
\end{proposition}

\begin{proof}
The first assertion immediately derives from the continuity of $\xi$ and $\rho\restr{\C}$ (Proposition \ref{rhocont}), and the fact that $\beta$ is continuous everywhere but on the set $\D \setminus \{1\}$, where it is only left-continuous. Then, the second assertion follows by the fact that $R$ has countably many discontinuities (see \cite[Theorem 11.8]{Rooji} or \cite[Chapter Three, Section 34, Paragraph VII]{Kura}).
\end{proof}

It is worth remarking that the right-discontinuity of $R$ on $\D \setminus \{1\}$ is not simply due to the right-discontinuity of $\beta$ on that set, but it is instead also related to the fact that $\rho\restr{\C}$ is not continuously extendable, as stated in Proposition \ref{rhocont}.
In fact, as we will see in Proposition \ref{ClG}, for $x \in \D \setminus \{0,1\}$ the right limit $\lim_{y\,\searrow\,x} R(y)$ does not exist and even more 
\begin{equation*}
R(x)\notin\big[\mkern-1mu\liminf_{y\,\searrow\,x}R(y)\,,\,\limsup_{y\,\searrow\,x}R(y)\mkern1mu\big]\,.
\end{equation*}

For example:
\begin{eqnarray*}
R(1/2) \!\!\! &=& \!\!\! 2/3\,, \text{ \ while \ }
\liminf_{x\,\searrow\,1/2}R(x)=0 \text{ \ and \ } \limsup_{x\,\searrow\,1/2}R(x)=1/2,\\
R(1/4) \!\!\! &=& \!\!\! 1/3\,, \text{ \ while \ }
\liminf_{x\,\searrow\,1/4} R(x)=1/2 \text{ \ and \ } \limsup_{x\,\searrow\,1/4} R(x)=1.
\end{eqnarray*}

\medskip

In particular, $R$ is not a Darboux function, meaning that it does not satisfy the intermediate value property. Actually, it is not Darboux from the right for every $x \in \D \setminus \{1\}$ (see Corollary \ref{Darboux}). Thus, by Darboux's theorem, the integral function $\int_0^x \!R(t) dt$ is differentiable everywhere but on the set $\D \setminus \{1\}$, where it is only left-differentiable.

\medskip

By a classical characterization of Baire class 1 functions given by Lebesgue \cite{Lebesgue} (see also \cite[Section 4.4]{Bressoud} or \cite[Chapter Two, Section 31, Paragraph II]{Kura}), we know that for every $\varepsilon > 0$ there exists a countable closed covering $\{C_n\}_{n \geq 1}$ of $[0,1]$ such that the oscillation of $R$ on each $C_n$ is less than $\varepsilon$. We recall that the oscillation of a real function $f$ on a subset $S$ of its domain is defined as $O_f(S) = \sup_{x,y \in S}|f(x) - f(y)|$.

\medskip

Before going on, we want to provide an explicit construction of a covering $\{C_n\}_{n \geq 1}$ as above. First, we need to estimate the oscillation of $R$ on the dyadic intervals as follows.

\medskip

For every $x \in \{0,1\}^n$ and $n \geq 1$, we consider the cylinder $\cyl{x} \subset \B$
and the interval $\xi(\cyl{x}) = [0.x\per{0},0.x\per{1}] \subset [0,1]$, and observe that
\begin{equation}\label{Rcyl}
R(\xi(\cyl{x})) =
\left\{\vrule width0pt height16pt\right.\!\!\!
\begin{array}{ll}
[0,1] & \text{if $x = 0^n$}\\[2pt]
\xi(\{\rho(x\trunc{k-1}0\per{1})\} \cup \cyl{\rho(x)}) & \text{if $x \neq 0^n$}
\end{array},
\end{equation}
where $k$ is the index of the last occurrence of 1 in $x$.
In particular, in the latter case we have $R(\xi(\cyl{x})) \subset \xi([\rho(x\trunc{k-1})])$, which implies that
\begin{equation}\label{Ocyl}
O_R(\xi(\cyl{x})) \leq 1/2^h \ \ \text{with} \ \ h = |\rho(x\trunc{k-1})| = |x|_1 - 1\,.
\end{equation}

Now we can start with our construction.
Given $\varepsilon > 0$, we choose a positive integer $\ell$ such that $1/2^{\ell-1} < \varepsilon$, and consider the set $D_\ell \subset \D$ consisting of all dyadic rational $x$ whose binary expansion $\beta'(x)$ contains at most $\ell$ 1's, that is $|\beta'(x)|_1 \leq \ell$. Then, the ordered set $(D_\ell,\geq\mkern2mu)$ is isomorphic to the countable ordinal $\omega^\ell+1$, and hence $D_\ell$ has Cantor-Bendixson rank $\ell+1$. Moreover, considering $D_\ell$ with the standard order in $\mathbb R$, we have $\min D_\ell = 0$ and $\max D_\ell = 0.1^\ell = \xi(1^\ell)$, and for every $x \in D_\ell \setminus \{0\}$, the immediate predecessor of $x$ in $D_\ell$, that is the largest element of $D_\ell$ smaller than $x$, is $p(x) = \xi(\beta(x)\trunc{k}\per{0})$, where $k$ denotes the index of the $\ell$-th 1 in $\beta(x)$. Keeping the same notations, we put
\begin{equation*}
C_x = \left\{\vrule width0pt height24pt\right.\!\!\!
\begin{array}{ll}
\xi(\cyl{1^\ell}) = [\xi(1^\ell\per{0}),\xi(1^\ell\per{1})] = [0.1^\ell,1] & \text{if $x = 1$}\\[2pt]
\xi(\cyl{\beta(x)\trunc{k}}) = [\xi(\beta(x)\trunc{k}\per{0}),\xi(\beta(x)\trunc{k}\per{1})] = [p(x),x] & \text{if $x \in D_\ell \setminus \{0\}$}\\[2pt]
\{0\} & \text{if $x=0$}
\end{array}.
\end{equation*}

Then, $\{C_x\}_{x \in D_\ell \cup \{1\}}$ is a countable closed covering of $[0,1]$ with the wanted oscillation bound. Indeed, in all cases, we have $O_R(C_x) \leq 1/2^{\ell-1} < \varepsilon$ by \eqref{Ocyl}.

\medskip

Concerning the fibers, it immediately follows from \eqref{R} that
\begin{equation}\label{R-1}
R^{-1}(y) \cap (0,1] = \beta^{-1}(\rho^{-1}(\xi^{-1}(y))) = \xi(\rho^{-1}\{\beta(y),\beta'(y)\}))
\end{equation}
for every $y \in [0,1]$. This actually coincides with $R^{-1}(y)$ if $y \neq 2/3$, while 0 must be added to get $R^{-1}(y)$ if $y = 2/3$.

\medskip

Taking into account equation \eqref{fibereq} and the properties of the maps $\beta$ and $\beta'$ we have seen at the end of Section \ref{setting}, we have
\begin{equation}\label{R-1bis}
R^{-1}(y) = \left\{\vrule width0pt height30pt\right.\!\!\!
\begin{array}{ll}
\xi(\fibexp{\sigma(\beta(y))}) & \text{if $y \in [0,1] \setminus (\D \cup \{2/3\})$ or $y = 1$}\\
\xi(\fibexp{\sigma(\beta(y))}) \cup \{0\} & \text{if $y = 2/3$}\\
\xi(\fibexp{\sigma(\beta(y))}) \cup \xi(\fibexp{\sigma(\beta'(y))}) & \text{if $y \in \D \setminus \{0,1\}$}\\
\xi(\fibexp{\sigma(\beta'(y))}) & \text{if $y = 0$}
\end{array},
\end{equation}
where the following notation is used for $b \in \C$
\begin{equation}\label{fibexp}
\fibexp{b} = \{\binexp{a}{b} \st a \in \A\}.
\end{equation}

\medskip

Since $\xi\restr{\C}$ is injective, the fiber $R^{-1}(y)$ contains uncountably many elements for every $y \in [0,1]$.
Moreover, according to Proposition \ref{fiber}, any fiber $R^{-1}(y)$ admits a maximum element
\begin{equation}\label{S}
S(y) = \left\{\vrule width0pt height22pt\right.\!\!\!
\begin{array}{ll}
\xi(\sigma(\beta(y))) & \text{if $y \in [0,1] \setminus \D$ or $y = 1$}\\
\max\{\xi(\sigma(\beta(y))),\xi(\sigma(\beta'(y)))\} & \text{if $y \in \D \setminus \{0,1\}$}\\
\xi(\sigma(\beta'(y))) & \text{if $y = 0$}
\end{array}.
\end{equation}
The map $S: [0,1] \to [0,1]$ defined by the equation above is a null measure section for $R$, as we will see in Proposition \ref{section}. Here, we limit ourselves to consider the following continuity properties of it.

\begin{proposition}\label{contS}
The map $S$ is continuous everywhere but on the set $\D \setminus \{0,1\}$, where it is either left- or right-continuous, with different existing left- and right-limits. Hence, $S$ is a Borel function of Baire class $1$.
\end{proposition}

\begin{proof}
The first assertion easily follows from the continuity properties of $\xi\,,\,\sigma\restr{\B}$ (Proposition \ref{sigmacont}), $\beta$ and $\beta'$, and the fact that, for every $y = (2k+1)/2^n \in \D \setminus \{0,1\}$, we have
\begin{equation*}
|\xi(\sigma(\beta(y)))-\xi(\sigma(\beta'(y)))| \leq 2\,d(\sigma(\beta(y)),\sigma(\beta'(y)))
\leq 2\,d(\beta(y),\beta'(y)) = 1/2^{n-1}.
\end{equation*}
Then, we can apply \cite{Rooji} as above to get the second assertion.
\end{proof}

Recalling how $\rho$ and $\widetilde\rho$ act on the binary digits of $x$ we readily deduce the next proposition.

\begin{proposition}The map $R$ satisfies the functional equations for every $x \in (0,1]$
\begin{equation}\label{functional1}
R\Big(\frac{x}{2}\Big)=1-R(x)\,,
\end{equation}
\begin{equation}\label{functional2}
R\Big(\frac{x+1}{2}\Big) =\frac{1-R(x)}{2}\,.
\end{equation}
\end{proposition}

\begin{proof}
For every $x \in (0,1]$, we have
\begin{equation*}
R\Big(\frac{x}{2}\Big) = \xi\Big(\rho\Big(\beta\Big(\frac{x}{2}\Big)\Big)\Big) = \xi(\rho(0\beta(x))) = \xi(\widetilde\rho(\beta(x))) = 1 - \xi\rho(\beta(x))) = 1 - R(x)\,,
\end{equation*}
\begin{equation*}
R\Big(\frac{x+1}{2}\Big) = \xi\Big(\rho\Big(\beta\Big(\frac{x+1}{2}\Big)\Big)\Big) = \xi(\rho(1\beta(x))) = \xi(0\widetilde\rho(\beta(x))) = \frac{1 - \xi\rho(\beta(x)))}{2} = \frac{1 - R(x)}{2}\,.\ \ 
\end{equation*}
\vskip-\lastskip
\vskip-1.2\baselineskip
\end{proof}
\vskip12pt

Notice that the relations \eqref{functional1} and \eqref{functional2} are verified by any map defined by means of a generalized substitution of type \eqref{rho}, when odd-indexed 1's go into a word $w$ and even-indexed 1's go into its complementary word $\widetilde{w}$.

\medskip

By using the relations \eqref{functional1} and \eqref{functional2}, one can get a countable family of functional equations, one for each word $w\in\{0,1\}^\infty$. Namely, if $w$ is a finite binary word such that $|w|=n$, $\rho(w)=v$ and $|v|=m$, for every $x\in (0,1]$ we have
\begin{equation*}
R\Big(\frac{x}{2^n} +\sum_{i=1}^n \frac{w_i}{2^i}\Big) = 
\sum_{i=1}^m \frac{v_i}{2^i} + \frac{1}{2^m}\Big(\frac{1+(-1)^{n+1}}{2}+(-1)^n R(x)\Big).
\end{equation*}

Moreover, from \eqref{functional1} and \eqref{functional2} it easily follows that
\begin{equation}\label{Reqn}
R(x) = \left\{\vrule width0pt height28pt\right.\!\!\!
\begin{array}{ll}
\displaystyle
2\,R\Big(x+\frac{1}{2}\mskip2mu\Big) & \text{if $x \in (0,1/2]$}\\[12pt]
\displaystyle
\frac{1}{2}\, R\Big(x-\frac{1}{2}\mskip2mu\Big) & \text{if $x \in (1/2,1]$}
\end{array}.
\end{equation}


\section{Graph properties}
\label{graph}

In this section we consider some properties of the map $R$ related to the geometry of its graph $G = \{(x,R(x)) \st x \in [0,1]\} \subset [0,1]^2$, which is depicted in Figure \ref{grafico} below. 

\medskip

We remark that there is a priori no reason to think that the dots in the picture represent a numerical approximation of points of the graph of $R$, as numerical computation is of course based on (finite) 
binary expressions of dyadic rational numbers, whereas $\rho$ acts always on the infinite
representation and is discontinuous precisely on dyadic rationals. However, the picture still 
represents a numerical approximation of the graph of $R$ for a somewhat deeper reason, that is because we can interpret it as the plot of an element of a sequence of step functions, constant on cylinder sets sharing a common finite prefix, converging to $R$. That this ``numerical'' convergence is mathematically meaningful is indeed ensured by Proposition 3.1, which establishes that $R$ can be point-wise limit of \emph{continuous} functions.

\medskip

\begin{figure}[h]
\centering
\includegraphics[scale=0.15]{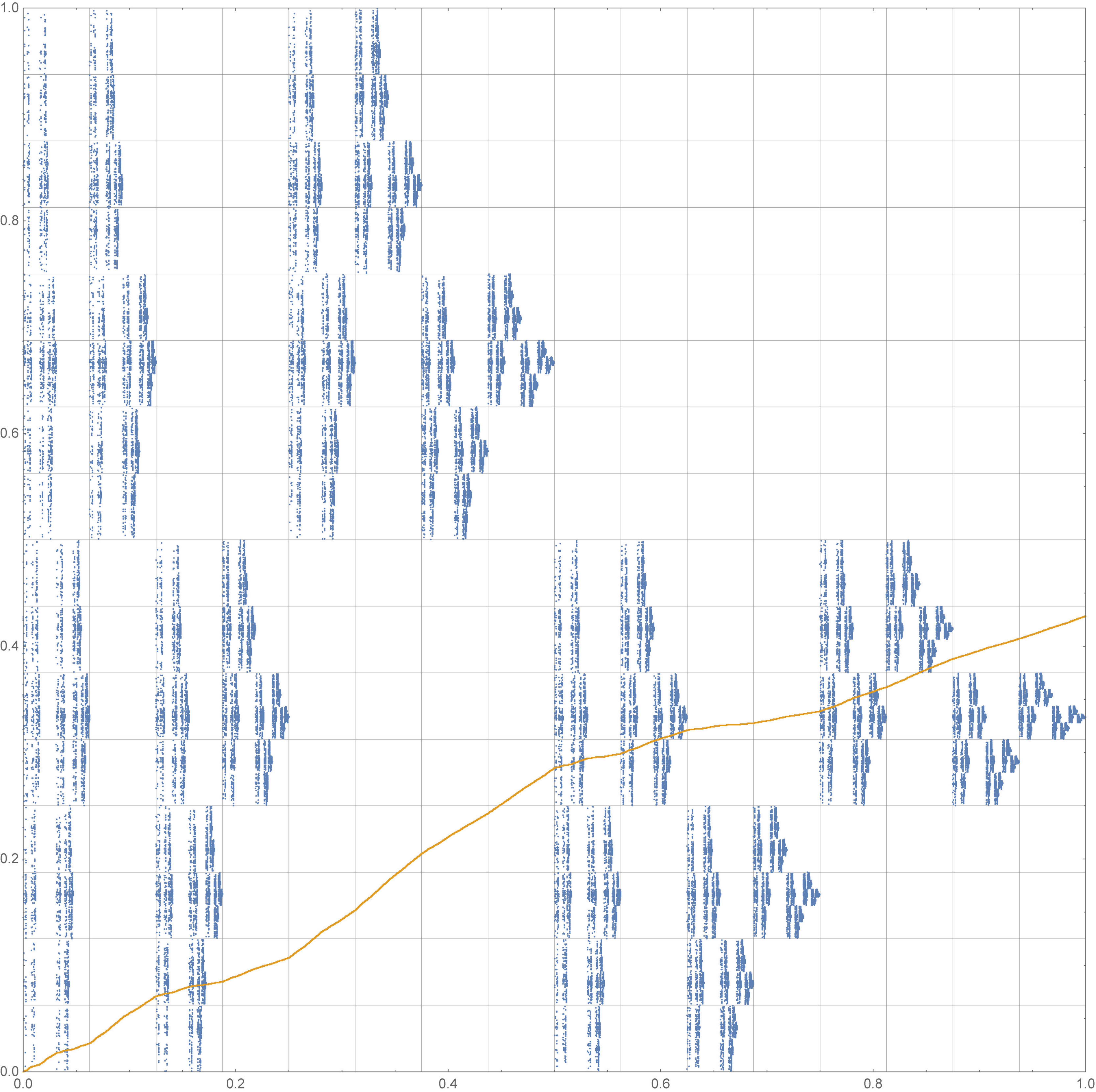} 
\caption{The graph of the map $R$ (blue) and of its integral function (orange).}
\label{grafico}
\end{figure}

\medskip

The functional relations \eqref{functional1} and \eqref{functional2} induce an interesting recursive structure on $G$.

\medskip

Consider, the transformations $T_0,T_1: [0,1]^2 \to [0,1]^2$ defined by the equations
\begin{equation*}
T_0(x,y) = \Big(\frac{x}{2}\,,1-y\Big) \quad \text{and}\quad T_1(x,y) = \Big(\frac{x+1}{2}\,,\,\frac{1-y}{2}\Big).
\end{equation*}

Let $\T$ be the operator which associates to each subset $E \subset [0,1]^2$ the subset
\begin{equation*}
\T(E) = T_0(E) \cup T_1(E).
\end{equation*}

Of course, $\T$ preserves compactness. Moreover, \eqref{functional1} and \eqref{functional2} imply the equality $\T(H) = H$ for $H = G - \{(0,2/3)\}$. Indeed, for every $x \in (0,1]$, we have
\begin{equation*}
T_0(x,R(x)) = \Big(\frac{x}{2}\,,\, 1 - R(x)\Big) = \Big(\frac{x}{2}\,,\,R\Big(\frac{x}{2}\Big)\Big)
\end{equation*}
and
\begin{equation*}
T_1(x,R(x)) = \Big(\frac{x+1}{2}\,,\, \frac{1 - R(x)}{2}\Big) = \Big(\frac{x+1}{2}\,,\,R\Big(\frac{x+1}{2}\Big)\Big).
\end{equation*}
Therefore,
\begin{equation*}
T_0(H) = H \cap ([0,1/2] \times [0,1]) \quad\text{and}\quad T_1(H) = H \cap ([1/2,1] \times [0,1]),
\end{equation*}
which gives $\T(H)=H$.

\medskip

We inductively define $K_n \subset [0,1]^2$ for $n \geq 0$, by putting
\begin{equation*}
K_0 = [0,1]^2 \text{ \ and \ } K_{n+1} = \T(K_n) = T_0(K_n) \cup T_1(K_n),
\end{equation*}
where 
\begin{equation}\label{TK}
T_0(K_n) = K_{n+1} \cap ([0,1/2] \times [0,1]) \quad \text{and} \quad T_1(K_n) = K_{n+1} \cap ([1/2,1] \times [0,1]).
\end{equation}
The trivial inclusion $K_1 \subset K_0$ implies that $K_{n+1} \subset K_n$ for every $n \geq 0$. So, the $K_n$'s form a decreasing sequence of compact subspaces of $[0,1]^2$ converging to
\begin{equation*}
K_\infty = \cap_{\,n\geq0\,}K_n \subset [0,1]^2.
\end{equation*}

\begin{figure}[htb]
\centering
\includegraphics[scale=0.25]{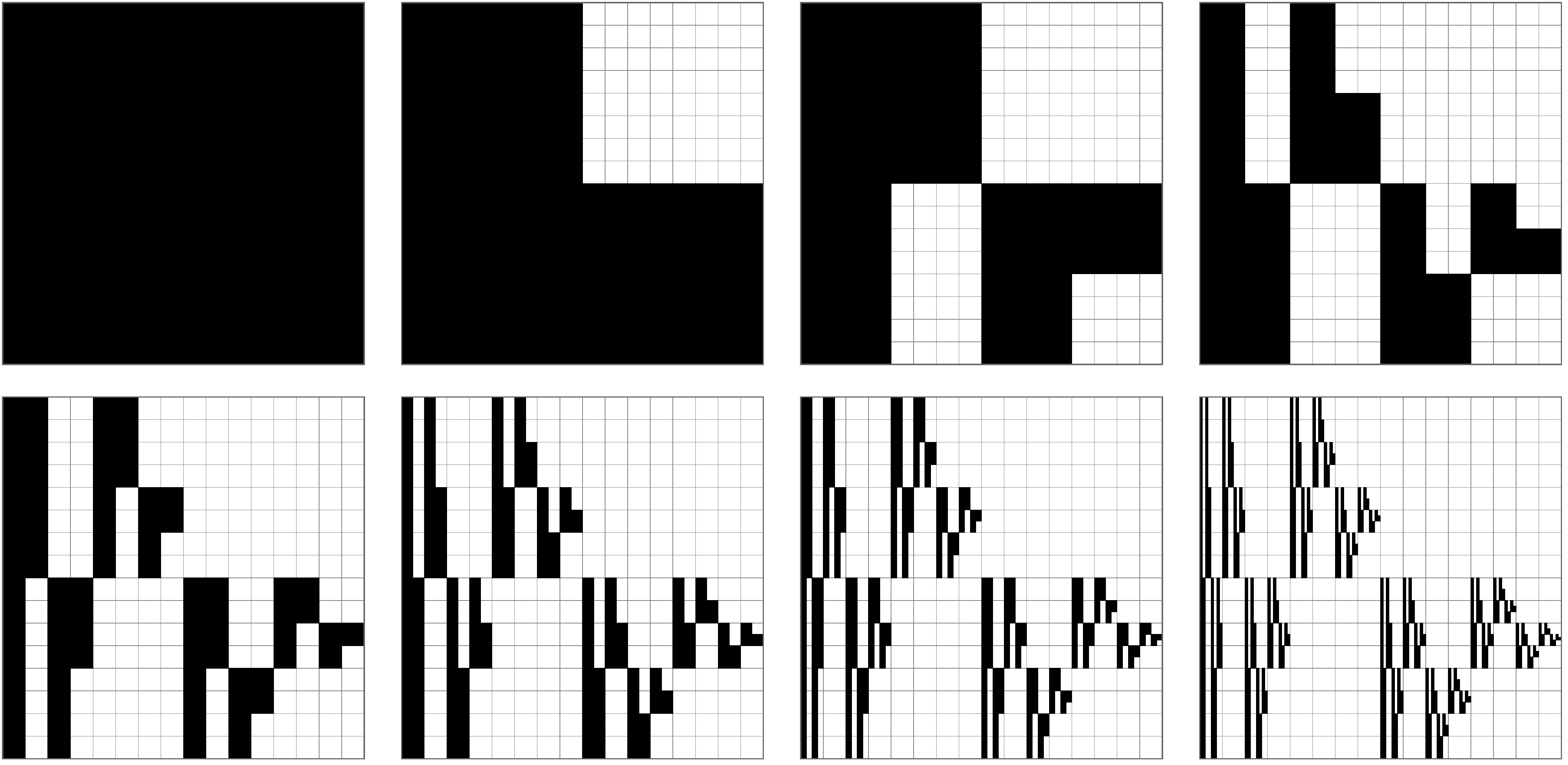} 
\caption{$K_n$ for $n= 0, \dots, 7$.}
\label{rettangoli}
\end{figure}

In order to describe the sequence of the $K_n$'s from a different point of view, we consider the maps
$t_0,t_1:[0,1] \to [0,1]$ defined by 
\begin{equation*}
t_0(y) = 1-y \quad \text{and} \quad t_1(y) = \frac{1-y}{2},
\end{equation*}
which give the action of $T_0$ and $T_1$ on the second coordinate.

\pagebreak 

\medskip

Moreover, we consider the intervals
\begin{equation*}
I_0=[0,1] \quad \hbox{and} \quad I_x = t_{x_1}\!\circ t_{x_2}\!\circ\cdots \circ t_{x_n}([0,1])\quad \hbox{for} \quad x \in \D \setminus \{1\}
\end{equation*}
where $0.x_1x_2 \dots x_n$ is a finite binary expression of $x$. Since the terminal 0's in the binary word $x_1x_2\dots x_n$ do not affect the interval $I_x$, this is well-defined, depending only on $x$.

\medskip

We point out that the interval $I_x$ satisfies the identities
\begin{equation}\label{due}
t_0(I_{x}) = I_{x/2} \quad \text{and} \quad t_1(I_{x}) = I_{(x+1)/2}\,.
\end{equation}
Then, by reasoning inductively as above, the interval $I_x$ for $x=0.x_1 \dots x_n$ can be shown to coincide with the cylinder generated by the word $\rho (x_1\dots x_n)$, that is
\begin{equation}\label{Ix}
 I_x= \cyl{\mkern1mu\rho (x_1\dots x_n)}\,.
\end{equation}

\medskip

For every $n \geq 0$, we can decompose $K_n$ as a union of rectangles as follows
\begin{equation*}
K_n=\cup_{k=0,\dots,2^n-1} [k/2^n,(k+1)/2^n] \times I_{k/2^n}.
\end{equation*}
This equality is trivially true for $n =0$, while it can be proved by induction for $n>0$ by taking into account of \eqref{due},
from which
\begin{equation*}
T_0(K_n)=\cup_{k=0,\dots,2^n-1} [k/2^{n+1},(k+1)/2^{n+1}] \times I_{k/2^{n+1}}
\end{equation*}
and
\begin{equation*}
T_1(K_n)=\cup_{k=2^n,\dots,2^{n+1}-1} [k/2^{n+1},(k+1)/2^{n+1}] \times I_{k/2^{n+1}}.
\end{equation*}

\medskip

Then, we can obtain $K_{n+1}$ from $K_n$ by replacing each rectangle
\begin{equation*}
[k/2^n,(k+1)/2^n] \times I_{k/2^n}
\end{equation*}
by the union
\begin{equation*}
([k/2^n,(2k+1)/2^{n+1}] \times I_{k/2^n}) \cup ([(2k+1)/2^{n+1},(k+1)/2^n] \times I_{(2k+1)/2^{n+1}}).
\end{equation*}

\begin{remark}\label{3/4}
Passing from $K_n$ to $K_{n+1}$ amounts to split each rectangle $[k/2^n,(k+1)/2^n] \times I_{k/2^n}$\break into four congruent closed sub-rectangles and remove the right-top or right-bottom one, depending on $n$ being even or odd respectively, while keeping the union of the other three sub-rectangles.
\end{remark}

Now, come back to $H$ and $G$. The trivial inclusion $H \subset K_0$ and the equality $\T(H) = H$ imply that $H \subset K_n$ for every $n \geq 0$. On the other hand, as
\begin{equation*}
\{0\} \times [0,1]= T_0(\{0\} \times [0,1])\subset\T(\{0\} \times [0,1]),
\end{equation*}
we also have $\{0\} \times [0,1]\subset K_n$ for every $n \geq 0$. Hence, $G \subset (\{0\} \times [0,1])\cup H$ implies $G\subset K_{\infty}$, and therefore $\Cl{G} \subset K_\infty$. We want to show that the reversed inclusion holds as well. In fact, we have the following slightly stronger result.

\begin{proposition}\label{ClG}
The following equalities hold
\begin{equation*}
K_\infty = \Cl{G} = G \cup(\cup_{x\in\D \setminus \{1\}}(\{x\} \times I_x)).
\end{equation*}
Moreover, for every $x\in\D \setminus \{0,1\}$,
\begin{equation*}
R(x) \notin I_x = \big[\mkern-1mu\liminf_{\xi\,\searrow\,x} R(\xi),\limsup_{\xi\,\searrow\,x} R(\xi)\mkern1mu\big].
\end{equation*} 
\end{proposition}

In order to prove the proposition a technical lemma is needed.

\begin{lemma} For every $x \in (0,1)$ we have
\begin{equation*}
K_\infty \cap (\{x\} \times [0,1]) = 
\left\{\vrule width0pt height15pt\right.\!\!\!
\begin{array}{ll}
\{x\} \times (\{R(x)\} \cup I_x) \text{\rm\ with } R(x) \notin I_x & \text{\rm if $x\in\D$}\\[2pt]
\{(x,R(x))\} & \text{\rm if $x\notin\D$}\end{array}.
\end{equation*}
\end{lemma}

\begin{proof}
The case when $x\in\D$ immediately follows from the definition of $K_\infty$ and the following statement: if $0.x_1\dots x_k$ is the shortest binary expression of $x$, then for every $n \geq k$
\begin{equation*}
K_n \cap (\{x\} \times [0,1]) = \{x\} \times (I_x \cup J_{x,n})
\end{equation*}
where $(J_{x,n})_{n\geq k}$ is a non-increasing sequence of intervals such that $\cap_{n\geq k} J_{x,n} = \{R(x)\}$ with $R(x) \not\in I_x$. We prove this statement by induction on $k \geq 1$.

\medskip

If $k=1$ then $x = 1/2$, and for every $n > 1$, we have
\begin{equation*}
K_n \cap (\{1/2\} \times [0,1]) = T_0(K_{n-1} \cap (\{1\} \times [0,1])) \cup T_1(K_{n-1} \cap (\{0\} \times [0,1])),
\end{equation*}
so that, by recalling equation \eqref{TK},
\begin{eqnarray*}
K_n \cap (\{1/2\} \times [0,1]) 
&\!\!\!=\!\!\!& T_0(\{1\} \times t_1^{n-1}([0,1])) \cup T_1(\{0\} \times [0,1])\\
&\!\!\!=\!\!\!& \{1/2\}\times (t_0\circ t_1^{n-1}([0,1]) \cup [0,1/2]).
\end{eqnarray*}
Then, $J_{1/2,n} = (t_0\circ t_1^{n-1}([0,1]))_{n\geq 1}$ is a non-increasing sequence of intervals whose intersection is $t_0(R(1)) = R(1/2) = 2/3$. Moreover, $J_{1/2,n} \cap I_{1/2} = J_{1/2,n} \cap [0,1/2] = \emptyset$ for every $n$ large enough, and hence $R(1/2) \not\in I_{1/2}$.

\medskip

Consider now the case of a dyadic rational $x = 0.x_1 \dots x_k$ with $k > 1$.
Let $x' = 0.x_2 \dots x_k = 2x \;(\text{mod}\;1)$ and assume, by the inductive hypothesis, that $K_n \cap (\{x'\} \times [0,1]) = \{x'\} \times (I_{x'} \cup J_{x',n})$ for every $n \geq k-1$, with $(J_{x',n})_{n\geq k-1}$ a non-increasing sequence of intervals whose intersection is $\{R(x')\}$ and $R(x') \notin I_{x'} = t_{x_2}\!\circ\dots\circ t_{x_k}([0,1])$. Then, for every $n \geq k$ we have
\begin{equation*}
K_n \cap (\{x\} \times [0,1]) = T_{x_1}(K_{n-1} \cap (\{x'\} \times [0,1])),
\end{equation*}
and hence
\begin{equation*}
K_n \cap (\{x\} \times [0,1]) = \{x\} \times (t_{x_1}(I_{x'}) \cup t_{x_1}(J_{x',n-1})).
\end{equation*}
Here, $(J_{x,n} = t_{x_1}(J_{x',n-1}))_{n \geq k}$ is a non-increasing sequence of intervals whose intersection is $\{R(x)\} = \{t_{x_1}(R(x'))\}$, and $R(x) \notin I_x= t_{x_1}(I_{x'})$. 

\medskip

We are left to consider the case when $x \notin\D$. For every $n \geq 1$, let $x\trunc{n} = 0.x_1 \dots x_n$ the $n$-th trucantion of the binary expression of $x$. We want to prove, by induction on $n \geq 1$, that
\begin{equation*}
K_n \cap (\{x\} \times [0,1]) = \{x\} \times I_{x\trunc{n}}.
\end{equation*}
In fact, for $n=1$ we have $I_{x\trunc{1}} = [0,1]$ if $x_1=0$ and $I_{x\trunc{1}} = [0,1/2]$ if $x_1=1$. In both cases, $K_1 \cap (\{x\} \times [0,1]) = \{x\} \times I_{x\trunc{1}}$ and $R(x) \in I_{x\trunc{1}}$. Moreover, for $n > 1$ we have 
\begin{equation*}
I_{x_n} = \left\{\vrule height24pt width0pt\right.\!\!\!
\begin{array}{ll}
I_{x\trunc{n-1}} & \text{if $x_n = 0$}\\[2pt]
I_{x\trunc{n-1}}^+ & \text{if $x_n = 1$ and $n$ is even}\\[2pt]
I_{x\trunc{n-1}}^- & \text{if $x_n = 1$ and $n$ is odd}\end{array},
\end{equation*}
where $I_{x\trunc{n-1}}^+$ and $I_{x\trunc{n-1}}^-$ denote the upper half and the lower half of $I_{x\trunc{n-1}}$, respectively. In all the cases, the equality $K_n \cap (\{x\} \times [0,1]) = \{x\} \times I_{x\trunc{n}}$ and the relation $R(x) \in I_{x\trunc{n}}$ follow by induction on $n$, taking into account that 
\begin{equation*}
\begin{array}{ll}
0.x_1\dots x_{n-1} < x < 0.x_1\dots x_{n-1}1 & \text{if } x_n = 0\,,\\[3pt]
0.x_1\dots x_{n-1}1 < x < 0.x_1\dots x_{n-1}\per{1} & \text{if } x_n = 1\,.
\end{array}
\end{equation*}
Then,
\begin{equation*}
K_\infty \cap ({x} \times [0,1]) = \{x\} \times (\cap_{n\geq 1} I_{x\trunc{n}}) = \{(x,R(x))\}.
\end{equation*}
\vskip-\lastskip
\vskip-\baselineskip
\end{proof}
\vskip0pt

\begin{proof}[Proof of Proposition \ref{ClG}]
According to the lemma above, taking the union over $x\in [0,1)$ we immediately get the following equality
\begin{equation} \label{qinf}
K_\infty = G \cup(\cup_{x\in\D\setminus \{1\}} (\{x\} \times I_x)),
\end{equation}
with
\begin{equation*}
G \cap(\cup_{x\in\D\setminus \{1\}} (\{x\} \times I_x))\ = \{(0,2/3)\}.
\end{equation*}

\medskip

Equation \eqref{qinf} implies that $\Cl{G} \subset K_\infty$. Therefore, in order to conclude that $K_\infty = \Cl{G}$, it suffices to show that the inclusion $\{x\} \times I_x \subset \Cl{G}$ holds for every $x = 0.x_1 \dots x_n \in \D \setminus \{1\}$. According to equation \eqref{Ix}, for a generic point $(x,y)\in \{x\}\times I_{x}$ we have $y=0.x_1\dots x_n v$, with $v\in \{0,1\}^\infty$. Then, the point $(x,y)$ can be arbitrarily approximated by a point $(x',R(x'))\in G$ with $x' = 0.x_1 \dots x_n(00)^kw$, $k$ sufficiently large and $w\in \{0,1\}^\infty$ any binary word such that $\rho (w) = v$. This concludes the proof of the first part of the proposition.

\medskip

To prove the second part of the proposition, we start by observing that the argument just exposed immediately entails that
\begin{equation*}
I_x \subset \big[\mkern-1mu\liminf_{\xi\,\searrow\,x} R(\xi),\limsup_{\xi\,\searrow\,x} R(\xi)\mkern1mu\big]
\end{equation*}
The opposite inclusion readily follows by equation \eqref{Ix}. 
\end{proof}

\begin{corollary}\label{Darboux}
The map $R$ is not Darboux from the right $($in the sense of {\rm\cite{Csaszar}}$)$ at any dyadic rational $x \in \D \setminus \{1\}$  and its graph $G$ is totally disconnected.
\end{corollary}

\begin{proof}
This is an immediate consequence of Proposition \ref{ClG}.
\end{proof}

Using the inclusion $G \subset K_\infty$, we can prove the next proposition.
\medskip

\begin{proposition}
\strut\kern8.5em
$\displaystyle\int_0^1 \!\! R(x)dx = \frac{3}{7}\,.$
\end{proposition}

\begin{proof}
By Remark \ref{3/4}, $\text{Area}\,K_n = \frac{3}{4} \text{Area}\,K_{n-1}$ for every $n>0$.
Then, taking into account that $\text{Area}\,K_0 = 1$, we get
\begin{equation}\label{3/4eqn}
\text{Area}\,K_n = \Big(\,\frac{3}{4}\,\Big)^n
\end{equation}
for every $n \geq 0$. Since this vanishes for $n \to \infty$, we can approximate the integral of $R$ by the integral $A_n$ of the piecewise constant function $R_n$ whose graph is given by the bottom edges of all the rectangles forming $K_n$. Now, $R_{2n+1}=R_{2n}$ for every $n \geq 1$, so we can write
\begin{equation*}
\int_0^1\!\!R(x)dx = \lim_{n\to \infty}A_{2n}.
\end{equation*}
Looking at what happens inside each rectangle of $K_{2n-2}$ when passing from $R_{2n-2}$ to $R_{2n}$ (compare $K_0$ and $K_2$ in Figure \ref{rettangoli}), we have
\begin{equation*} 
A_{2n} - A_{2n-2} = \frac{3}{16}\, \text{Area}\,K_{2n-2} = \frac{3}{16} \cdot \Big(\frac{9}{16} \Big)^{n-1}
\end{equation*}
Therefore, starting from $A_0=0$ we get
\begin{equation*}
\lim_{n \to \infty}A_{2n} = \frac{3}{16} \,\sum_{n\,=\,0}^{\infty}\Big(\frac{9}{16}\Big)^n = \frac{3}{16} \cdot \frac{16}{7} = \frac{3}{7}.
\end{equation*}
\vskip-\lastskip
\vskip-1.7\baselineskip
\end{proof}
\vskip12pt

Finally, we can compute the Box dimension $\dim_B G$ and estimate the Hausdorff dimension $\dim_H G$ as follows.

\begin{proposition}
\strut\kern2.5em
$\displaystyle 1 \leq \dim_H G \leq \dim_B G = \dim_B K_\infty = \log_23.$
\end{proposition}

\begin{proof}
The two inequalities derive from the projection of $G$ onto $[0,1]$ and from the general 
relation between the Hausdorff dimension and the box-dimension, respectively. The first equality comes from the invariance of the box-dimension under closure. The second equality follows from the fact that the boxes of the form $[k/2^n,(k+1)/2^n] \times [\ell/2^n,(\ell+1)/2^n]$ needed to cover $K_\infty$ are exactly the ones contained in $K_n$, which are $3^n$, as it is clear by equation \eqref{3/4eqn}.
\end{proof}


\section{Analytical properties}
\label{analysis}

Let us start with some global properties of the map $R$ with respect to the Lebesgue measure. 
We indicate by $\lambda(A)$ the standard Lebesgue measure of a measurable set $A\subset [0,1]$.

\medskip

In the following propositions we indicate by $\N \subset [0,1]$ the subset of 2-normal numbers, that is the numbers $x \in [0,1]$ which admit a binary expansion where all the binary sequences of any given length $k \geq 1$ occur with the same asymptotic relative frequency $1/2^k$. In particular, $\N$ does not contain any rational number, and hence $\beta(x) = \beta'(x)$ for every $x \in \N$. We recall that $\N$ has full measure $\lambda(\N)=1$, as proved in \cite{Borel}.

\medskip

For any infinite binary word $b \in \B$ and any finite binary word $w \in \{0,1\}^k$, we indicate by $f_n(w,b)$ the relative frequency of $w$ among all the $n$ subwords of $b\trunc{n+k-1}$ of length $k$, and by $f(w,b)$ the asymptotic relative frequency of $w$ in $b$, if it exists. Namely, we put
\begin{equation}\label{fn}
f_n(w,b) = \frac{|\{i=1,\dots,n \st w = b_ib_{i+1}\dots b_{i+k-1}\}|}{n}
\end{equation}
and
\begin{equation}\label{f}
f(w,b) = \lim_{n\to\infty} f_n(w,b).
\end{equation}

\medskip

Then, to say that $x = 0.x_1x_2\dots$ is normal means that $f(w,x_1x_2\dots) = 1/2^k$ for every $w \in \{0,1\}^k$.
This is equivalent to the property that for every $k \geq 1$, if $(w_1,w_2, \dots)$ is the decomposition of $x_1x_2\dots$ into contiguous blocks of length $k$, then the asymptotic relative frequency of any $w \in \{0,1\}^k$ in the sequence $(w_1,w_2, \dots)$ is $1/2^k$ (see \cite{Borel}, \cite{Bugeaud1}, \cite{Cassels}, \cite{Niven} and \cite{Pillai}).

\begin{proposition}\label{section}
The map $R$ is measurable but not bi-measurable. Moreover, the map $S$ defined by equation \eqref{S} is a null measure section of $R$, being $S([0,1]) \subset [0,1] \setminus \N$.
\end{proposition}

\begin{proof}
The map $R$ is almost everywhere continuous by Proposition \ref{Rcont}, hence it is measurable.
Concerning the section $S$, for every $y \in [0,1]$ the image $S(y)$ does not contain any pair of consecutive $0$'s in its binary expansion $\beta(S(y))$, which coincides with either $\sigma(\beta(y))$ or $\sigma(\beta'(y))$. Then, we have $S([0,1]) \subset [0,1] \setminus \N$, which implies that $\lambda(S([0,1]))=0$ by \cite{Borel}. Therefore, every set $A \subset [0,1]$ is the image $R(S(A))$ of the null measure set $S(A)$. In particular, this is true for any non-measurable subset of $[0,1]$, like the Vitali set. Thus, $R$ is not bi-measurable.
\end{proof}

We remark that, in the light the uncountability of all the fibers of $R$ discussed in Section \ref{function}, part of Proposition \ref{section} could be also obtained as a direct consequence of a general theorem by Purves \cite{Purves} stating that, if $f$ is a bi-measurable map from a standard Borel space $X$ to a Polish space $Y$, then at most countably many fibers of $f$ can be uncountable.

\begin{proposition}\label{sing}
The map $R$ is a singular Borel function, being $R(\N) \subset [0,1] \setminus \N$.
\end{proposition}

\begin{proof}
It is enough to prove the last inclusion, since it implies that $\lambda(R^{-1}([0,1] \setminus \N)) = 1$ with $\lambda([0,1] \setminus \N) = 0$, and hence the singularity of $R$. So, given $x = 0.x_1x_2\ldots \in \N$ and $R(x) = y = 0. y_1y_2\ldots$, with $y_1y_2\ldots = \rho(x_1x_2\dots)$, we have to prove that $y$ is not normal. Notice that $x_1x_2\ldots = \beta(x) \in \C$ and $y_1y_2\ldots \in \B$ are both infinite binary words (the latter coincides with either $\beta(y)$ or $\beta'(y)$).
Here, we think of the normality of $x$ in terms of asymptotic equi-distribution of contiguous blocks in $x_1x_2\dots$ of every given length, as discussed above.

By considering the blocks of length 2 in the word $x_1x_2\dots$, and taking into account the substitution rule $\tau$ defined in \eqref{tau} and the asymptotic equi-distribution of $01$ and $10$, guaranteed by the normality of $x$, we can immediately conclude that $R(x)$ is simply normal, meaning that $0$ and $1$ have the same asymptotic relative frequency $1/2$ in the word $y_1y_2\dots$.

By arguing instead on the blocks of length 8 in the word $x_1x_2\dots$, we will prove that
\begin{equation}\label{pairs}
\lim_{n \to \infty} (f_n(01,y_1y_2\dots) + f_n(10,y_1y_2\dots)) > 
\lim_{n \to \infty} (f_n(00,y_1y_2\dots)+f_n(11,y_1y_2\dots)),
\end{equation}
assuming that both the limits exist, otherwise $y = R(x)$ cannot be normal and we are done.
This prevents the word $y_1y_2\dots$ to satisfy the asymptotic equi-distribution of the binary subwords of length 2, and hence implies once again that $y = R(x)$ is not normal.

To prove \eqref{pairs}, let $\overline n \geq 1$ be arbitrarily fixed. Since $f(0,x_1x_1\dots)=f(1,x_1x_2\dots) = 1/2$, for every $\delta > 0$ there exists $m \geq 1$ such that
the length $n$ of the word $\rho(x_1x_2\dots x_{8m})$ satisfies the following properties
\begin{equation*}
\frac{n}{8m} = \frac{|\rho(x_1x_2\dots x_{8m})|}{8m} = \frac{|x_1x_2\dots x_{8m}|_1}{8m} \in (1/2 - \delta, 1/2 + \delta) \ \ \text{and} \ \ n > \overline n
\end{equation*}
We want to show that for a suitable choice of $m$ the following inequality holds
\begin{equation}\label{pairsbis}
f_n(01,y_1y_2\dots) + f_n(10,y_1y_2\dots) > 
f_n(00,y_1y_2\dots)+f_n(11,y_1y_2\dots) + c,
\end{equation}
where $c > 0$ is a constant independent on $n$, which gives \eqref{pairs}.

Let $(w_1,w_2,\dots,w_m)$ be the decomposition of the word $x_1x_2\dots x_{8m}$ in blocks of length 8, and $(v_1,v_2,\dots,v_m)$ be the corresponding decomposition of the word $y_1y_2\dots y_n$ in blocks of variable length $\leq 8$, with $v_i = \rho(w_i)$ for every $i = 1, \dots, k$.

Based on the normality of $x$, we can assume $m$ sufficiently large in such a way that the relative frequency of any binary word $w$ of length 8 in the sequence $(w_1,w_2,\dots,w_m)$ belongs to the interval $(1/2^8 - \delta,1/2^8 + \delta)$. Then, we consider the pairs of consecutive digits in $y_1y_2\dots y_n$, separating those occurring inside the blocks from those formed by the last digit of a block and the first digit of the next one (disregarding the empty blocks).

The pairs of the latter type are less than $m$. Concerning the pairs of the former type, direct inspection shows that in the images of all binary words of length 8 under $\rho$, the pairs 01 and 10 occur 541 times, while the pairs 00 and 11 occur 228 times. Therefore, inside the blocks $(v_1, v_2, \dots, v_m)$ the number of occurrences of the pairs 01 and 10 is at least $541(1/2^8 - \delta)m$, while the number of occurrences of the pairs 00 and 11 at most $228(1/2^8 + \delta)m$. Putting all together we have
\begin{eqnarray}
\label{f1}
f_n(01,y_1y_2\dots) + f_n(10,y_1y_2\dots) &\geq& \frac{541(1/2^8 - \delta)m}{n}
\geq \frac{541(1/2^8 - \delta)}{8(1/2+\delta)}\,,\\
\label{f2}
f_n(00,y_1y_2\dots) + f_n(11,y_1y_2\dots) &\leq& \frac{228(1/2^8 + \delta)m + m}{n}
\leq \frac{228(1/2^8 + \delta) + 1}{8(1/2-\delta)}\,.
\end{eqnarray}
At this point, it is enough to observe that \eqref{f1} is greater than \eqref{f2} for $\delta$ sufficiently small.
\end{proof}

As a consequence of Proposition \ref{sing}, the map $R$ does not preserve the Lebesgue measure. The Lebesgue measure of the inverse images of dyadic cylinders can be directly computed as follows.

\medskip

According to equation \eqref{R-1bis}, for every $y \in \{0,1\}^n$ and $n \geq 1$, we have
\begin{equation}\label{R-1cyl}
R^{-1}(\xi(\cyl{y})) \setminus \{0\} = \left\{\vrule height15pt width0pt\right.\!\!\!
\begin{array}{ll}
\cup_{b \in \B}\xi(\fibexp{\sigma(yb)}) & \text{if $y = 1^n$}
\\[3pt]
\cup_{b \in \B}\xi(\fibexp{\sigma(yb)}) \cup \xi(\fibexp{\sigma(y')}) & \text{otherwise}\end{array},
\end{equation}
where $y' = \beta(\xi(y)+1/2^n)$.

\medskip

Then, by putting $|\sigma(y)| = m$ and recalling that $|\sigma(y)|_1 = |y| = n$, we get
\begin{equation*}
\begin{array}{rcl}
\lambda(\cup_{b \in \B}\xi(\fibexp{\sigma(yb)}))
& \!\!\!=\!\!\! & 
\lambda(\{\binexp{a}{\sigma(yb)}\st a \in \A, b \in \B\})
= \lambda(\{\binexp{a}{\sigma(y)}b\st a \in \mathbb N^n,b\in\B\})\\[4pt]
& \!\!\!=\!\!\! & 
\lambda(\cup_{a\in \mathbb N^n}\cyl{\mkern-1mu\binexp{a}{\sigma(y)}})
= \textstyle
\sum_{a \in \mathbb N^n}\lambda(\cyl{\mkern-1mu\binexp{a}{\sigma(y)}})
= \textstyle
\sum_{a \in \mathbb N^n}1/2^{|\binexp{a}{\mkern1mu\sigma(y)}|}\\[4pt]
& \!\!\!=\!\!\! & 
\textstyle
\sum_{a \in \mathbb N^n}1/2^{(2a_1+\ldots+2a_n+m)}
= \textstyle
\big(\sum_{a_1 \geq 0}1/2^{2a_1}\big) \cdots \big(\sum_{a_n \geq 0}1/2^{2a_n}\big)/2^m\\[4pt]
& \!\!\!=\!\!\! & 
\textstyle
(4/3)^n/2^m = 2^{2n-m}/3^n \leq (2/3)^n.
\end{array}
\end{equation*}
Moreover, by taking into account the inclusions
\begin{equation*}
\xi(\fibexp{\sigma(y')}) \subset \cup_{b \in \B}\xi(\fibexp{\sigma(y'\trunc{k}b)}),
\end{equation*} 
which hold for every $k \geq 0$, and by applying the above formula for the measure of the right-side terms, we obtain
\begin{equation*}
\lambda(\xi(\fibexp{\sigma(y')})) = 0.
\end{equation*}
Therefore, based on \eqref{R-1cyl}, we have
\begin{equation}\label{muR-1cyl}
\lambda(R^{-1}(\xi(\cyl{y}))) = 2^{2n-|\sigma(y)|}\!/3^n \leq (2/3)^n
\end{equation}
for every $y \in \{0,1\}^n$ and $n \geq 1$.

\medskip

We remark that, by generalizing the above computation of the Lebesgue measure of the inverse images of dyadic cylinders, one could show that actually the map $R$ cannot preserve any product measure $\mu_p$ with $0 < p < 1$, determined by $\mu_p(\xi(\cyl{x})) = p^m(1-p)^{n-m}$ where $m = |x|_1$, for every $x \in \{0,1\}^n$ and $n \geq 1$.

\medskip

Moreover, Proposition \ref{sing} also entails that no $R$-invariant probability measure can be absolutely continuous with respect to Lebesgue measure. Indeed, if there were such a measure $\nu$, then we would have $\nu ([0,1] \setminus \N) = 0$, and hence $\nu(R^{-1}([0,1] \setminus \N)) = 0$. Thus, the density function of $\nu$ would vanish almost everywhere on a set containing $\N$.

\medskip

Now we pass to the topological and metric structure of the fibers of $R$, which we recall to be all uncountable, as a consequence of equation \eqref{R-1bis} in Section \ref{function}. In particular, we will compute the Hausdorff dimension $\dim_H R^{-1}(y)$ of the fibers of the rational numbers $y \in \Q$ and estimate that of all the other fibers. Let us start with some preliminary results.

\begin{proposition}\label{fibercantor}
For every $y \in [0,1]$, the closure $\Cl{R^{-1}(y)}$ of the fiber $R^{-1}(y)$ is a null measure Cantor set in $[0,1]$, with a countable remainder $\Cl{R^{-1}(y)} \setminus R^{-1}(y) \subset \D$.
\end{proposition}

\begin{proof}
First we prove the inclusion
\begin{equation}\label{closure}
\Cl{\xi(\fibexp{b})} \subset \xi(\fibexp{b}) \cup \D
\end{equation}
for every $b \in \C$, which implies that $\Cl{R^{-1}(y)} \setminus R^{-1}(y) \subset \D$ for every $y \in [0,1]$ by \eqref{R-1bis}.

Given $x \in \Cl{\xi(\fibexp{b})}$, let $(x_n)_{n \geq 1} \subset \xi(\fibexp{b})$ be a sequence such that $\lim_{n \to \infty} x_n = x$. Then, by the definition of $\fibexp{b}$ in \eqref{fibexp}, there exists a sequence $(a_n)_{n \geq 1} \subset \A$ such that $x_n = \xi(\binexp{a_n}{b})$, and hence $\binexp{a_n}{b} = \beta(x_n)$, for every $n \geq 1$. If $x \in \D$ we are done, otherwise the continuity of $\beta$ at $x$ (see Section \ref{setting}) implies that $\lim_{n \to \infty}\binexp{a_n}{b} = \beta(x) \in \C$. If the sequence of the $i$-th components $((a_n)_i)_{n \geq 1}$ is bounded for every $i \geq 1$, then by compactness we can replace the sequence $(a_n)_{n \geq 1}$ by a subsequence converging to $a \in \A$, and recalling that $\eta$ is a homeomorphism, we have
\begin{equation*}
\lim_{n \to \infty}\binexp{a_n}{b} = \lim_{n \to \infty}\eta(2a_n + \eta^{-1}(b)) = \eta(\lim_{n \to \infty}2a_n + \eta^{-1}(b)) = \eta(2a + \eta^{-1}(b)) = \binexp{a}{b}\,,
\end{equation*}
and hence $x = \xi(\binexp{a}{b}) \in \xi(\fibexp{b})$. If instead some sequence $((a_n)_i)_{n \geq 1}$ is unbounded, let $m$ be the minimum index $i \geq 1$ for which this happens. Replacing the sequence $(a_n)_{n \geq 1}$ by a suitable subsequence such that $(a_n)\trunc{m-1} = c$ for some $c \in \{0,1\}^{m-1}$ and every $n \geq 1$, we have
\begin{equation*}
\beta(x) = \lim_{n \to \infty}\binexp{a_n}{b} = \binexp{c}{b\trunc{m-1}}\per{0},
\end{equation*}
which is absurd, being $\beta(x) \in \C$. So this case cannot occur, and we finished the proof of \eqref{closure}.

At this point, for any $y \in (0,1]$, taking into account that $\{y\} = \cap_{n \geq 1}\xi(\cyl{\beta(y)\trunc{n}})$, and hence $R^{-1}(y) = \cap_{n \geq 1} R^{-1}(\xi(\cyl{\beta(y)\trunc{n}}))$, we obtain $\lambda(\Cl{R^{-1}(y)}) = \lambda(R^{-1}(y)) = 0$ by \eqref{muR-1cyl}. Similarly, we can obtain $\lambda(\Cl{R^{-1}(0)}) = \lambda(R^{-1}(0)) = 0$ starting from $\{0\} = \cap_{n \geq 1}\xi(\cyl{0^n})$.

We are left to prove that $\Cl{R^{-1}(y)}$ is a Cantor set for every $y \in [0,1]$. To this end, it is enough to show that $\Cl{R^{-1}(y)}$ is topologically 0-dimensional and perfect (see \cite[Chapter Four, Section 45, Paragraph II]{Kura}). The 0-dimensionality immediately follows from the fact that $\Cl{R^{-1}(y)}$ cannot contain any interval, having a null measure. For the perfectness, we observe that every $x \in \Cl{R^{-1}(y)} \setminus R^{-1}(y)$ cannot be isolated. On the other hand, if $x \in R^{-1}(y) \setminus \{0\}$, then $x = \xi(\eta(c))$ for some $c \in \C$ (see \eqref{R-1bis}), and we have $x = \lim_{n \to \infty}x_n$ with the sequence $(x_n = \xi(\eta(c_n)))_{n \geq 1} \subset R^{-1}(y)$ defined by $(c_n)_i = c_i + 2\delta_{i,n}$ for any $i,n \geq 1$. Finally, $0$ is a limit point for every fiber $R^{-1}(y)$, and in particular for $y = 2/3$. In fact, $0 = \lim_{n \to \infty}x_n$ where $(x_n)_{n \geq 1} \subset R^{-1}(y)$ is defined like above, starting with any $c \in \C$ such that $\xi(\eta(c)) \in R^{-1}(y)$ and putting $(c_n)_i = c_i + 2n\delta_{i,1}$ for any $i,n \geq 1$.
\end{proof}

In order to compare different fibers of $R$, we introduce a partial order relation between infinite binary sequences $b,c \in \C$ as follows
\begin{equation}\label{prec}
b \succcurlyeq c \ \Leftrightarrow \ \text{there exists $m \geq 0$ such that $|c\trunc{i}|_1 \leq |b\trunc{i}|_1 + m$ \ for every $i \geq 1$}.
\end{equation}
We also introduce the induced relation equivalence
\begin{equation}\label{equiv}
b \approx c \ \Leftrightarrow \ \text{$b \succcurlyeq c$ \ and \ $c \succcurlyeq b$}\,.
\end{equation}
The above relations provide a kind of uniform control on the distribution of the 1's in the words. In particular, $b \succcurlyeq c$ and $b \approx c$ imply analogous relations $f(1,b) \geq f(1,c)$ and $f(1,b) = f(1,c)$ respectively, between the asymptotic relative frequencies of 1's, if they exist. But of of course the opposite implications are false.

\medskip

The next two lemmas provide our technical tool for the comparison of fibers. Here, we recall that $\S$ stands for the set of all infinite binary words without any pair of consecutive $0$'s.

\begin{lemma}\label{phibc}
For every $b,c \in \C$,
the natural bijection $\phi_{b,c}:\fibexp{b} \to \fibexp{c}$ defined by
\begin{equation*}
\phi_{b,c}(\binexp{a}{b}) = \binexp{a}{c\mkern1mu}
\end{equation*}
with $a \in \A$, is a homeomorphism. Moreover, $\phi_{b,c}$ is a Lipschitz map if $b \succcurlyeq c$ and $c \in \S$, and hence $\phi_{b,c}$ is a bi-Lipschitz map if $b \approx c$ and $b,c \in \S$.
\end{lemma}

\begin{proof}
Given any $b \in \C$, we consider the natural bijection $\phi_b:\A \to \fibexp{b}$ defined by $\phi_b(a) = \binexp{a}{b}= \eta(2a + \eta^{-1}(b))$. This is a 1-Lipschitz homeomorphism, being the composition of the isometric map $\A \to \A$ which sends $a$ to $2a+\eta^{-1}b$ with the 1-Lipschitz homeomorphism $\eta: \A \to \C$. Moreover, for every $a,a' \in \A$ such that $a\trunc{n} = a'\trunc{n}$ and $a_{n+1}\neq a'_{n+1}$ we put
\begin{equation*}
h(a,a') = a_1 + \ldots + a_n + \min\{a_{n+1},a'_{n+1}\}\,,
\end{equation*}
and observe that the maximum index $i$ for which $\binexp{a}{b}\trunc{i} = \binexp{a'}{b}\trunc{i}$ is given by $2h(a,a') + k^b_n$, where $k^b_n$ denotes the index of the $n$-th occurrence of 1 in $b$.

Then, for every $b,c \in \C$ the identity $\phi_{b,c} = \phi_c \circ \phi_b^{-1}$ implies that 
$\phi_{b,c}$ is a homeomorphism.
Furthermore, if $b \succcurlyeq c$ and $c \in \S$ we can see that $\phi_{b,c}$ is a Lipschitz map as follows.

Let $m \geq 1$ such that $|c\trunc{i}|_1 \leq |b\trunc{i}|_1 + m$ for every $i \geq 1$. In particular, we have
\begin{equation*}
|c\trunc{k^b_i}|_1 \leq |b\trunc{k^b_i}|_1 + m = i + m,
\end{equation*}
and therefore
\begin{equation}\label{kbi}
k^b_i + 1 \leq k^c_{i+m+1} \leq k^c_i + 2m + 2
\end{equation}
for every $i \geq 1$, where $k^b_i$ denotes the index of the $i$-th occurrence of 1 in $b$ as above, and the last inequality derives from the fact that $c \in \S$ does not contain any pair of consecutive 0's. 

Given any two different sequences $a,a' \in \A$, let $n$ be the maximum index $i$ such that $a\trunc{i} = a'\trunc{i}$. Then, by the above observation, we have
\begin{equation*}
d(\binexp{a}{b},\binexp{a'}{b}) = 1/2^{2 h(a,a') + k^b_n+1}
\ \ \ \text{and} \ \ \
d(\binexp{a}{c},\binexp{a'}{c}) = 1/2^{2 h(a,a') + k^c_n+1},
\end{equation*}
from which by \eqref{kbi} we obtain
\begin{equation*}
\frac{d(\phi_{b,c}(\binexp{a}{b}),\phi_{b,c}(\binexp{a'}{b}))}
{d(\binexp{a}{b},\binexp{a'}{b})} \leq 2^{2m+1}.
\end{equation*}
So, we can conclude that $\phi_{b,c}$ is a Lipschitz map with Lipschitz constant $2^{2m+1}$.
\end{proof}

\begin{lemma}\label{psibc}
For every $b,c \in \C$, the bijection $\psi_{b,c} = \xi \circ \phi_{b,c} \circ \beta\restr{\xi(\fibexp{b})}:\xi(\fibexp{b}) \to \xi(\fibexp{c})$ given by\break
\vskip-21pt
\begin{equation*}
\psi_{b,c}(\xi(\binexp{a}{b}) = \xi(\binexp{a}{c})
\end{equation*}
\vskip3pt
with $a \in \A$, is a Lipschitz map if $b \succcurlyeq c$ and $b \in \S$, and hence $\psi_{b,c}$ 
is a bi-Lipschitz homeo\-morphism if $b \approx c$ and $b,c \in \S$.
\end{lemma}

\begin{proof}
In the light of the previous lemma, recalling that $\xi$ is 2-Lipschitz, it suffices to show that $\beta\restr{\xi(\fibexp{b})}$ is a Lipschitz map for every $b \in \S$.

Take any two different real numbers $x=\xi(\binexp{a}{b})$ and $x'=\xi(\binexp{a'}{b})$ in $\xi(\fibexp{b})$, with $a\trunc{n} = a'\trunc{n}$ and $a_{n+1} \neq a'_{n+1}$. Without loss of generality, we assume that $a_{n+1} < a'_{n+1}$. Like in the proof of the previous lemma, we have
$\binexp{a}{b}\trunc{2h(a,a')+k^b_n} = \binexp{a'}{b}\trunc{2h(a,a')+k^b_n}$. Denoting by $w$ this word, we can write $\binexp{a}{b} = w1c$ and $\binexp{a'}{b} = w00c'$, and hence $x = 0.w1c$ and $x'=0.w00c'$, for suitable $c,c' \in \C$.
Then, we get
\begin{equation*}
|x - x'| = 0.x1c - 0.x00c' \geq 0.x1 - 0.x01 = 0.0^{|w|+1}1 = 1/2^{|w|+2}
\end{equation*}
while
\begin{equation*}
d(\beta(x),\beta(x')) = d(\binexp{a}{b},\binexp{a'}{b}) = 1/2^{|w|+1}.
\end{equation*}
These inequalities immediately imply that $\beta\restr{\xi(\fibexp{b})}$ is a 2-Lipschitz map.
\end{proof}

Now we focus on the ``rational fibers'', that is the fibers $R^{-1}(y)$ with $y \in \Q$.

\begin{proposition}\label{ratfib}
If $y \in \Q$ then $\xi(\sigma(\beta(y))),\xi(\sigma(\beta'(y))) \in \Q$ when defined, and so $S(y) \in \Q$.
\end{proposition}

\begin{proof}
Since any binary expansion of a rational number is periodic, both $\beta(y)$ and $\beta'(y)$ have the form $w\per{p}$, where $w$ and $p$ are finite binary of length $n \geq 0$ and $\ell \geq 1$, respectively. Then, recalling that $\sigma$ can be interpreted as a true substitution rule applied to the first differences, $\sigma(\beta(y))$ and $\sigma(\beta'(y))$ have both the corresponding form $v\per{q}$, with
\begin{equation}\label{ratfibeq}
v = \left\{\vrule height14pt width0pt\right.\!\!\!
\begin{array}{ll}
\sigma(w) & \text{if $w_n = p_\ell$}\\
\sigma(wp) & \text{if $w_n \neq p_\ell$}
\end{array}
\ \ \ \text{and} \ \ \ \
q = \left\{\vrule height14pt width0pt\right.\!\!\!
\begin{array}{ll}
\sigma(p) & \text{if $p_\ell = 1$}\\
\sigma'(p) & \text{if $p_\ell = 0$}
\end{array},
\end{equation}
where we assume $w_n = 1$ for $n = 0$ and indicate by $\sigma'(p)$ the word $\sigma(p)$ with the first digit complemented. In all cases, we have a periodic binary sequence with a period $q$ of length $n$, and by applying $\xi$ we get rational number in $\Q$.
\end{proof}

The next proposition expresses the Hausdorff dimension $\dim_H R^{-1}(y)$ for $y \in \Q$ in terms of the density $d$ of the $1$'s in the period $q$ of $\sigma(\beta(y))$ or $\sigma(\beta'(y))$ as given by \eqref{ratfibeq}, that is 
\begin{equation*}
d = |q|_1/|q|\,.
\end{equation*}

We recall that $\beta(y)$ and $\beta'(y)$ coincide for $y \in [0,1] \setminus \D$, while only one of them is defined for $y = 0,1$. On the other hand, for $y \in \D \setminus \{0,1\}$ the period of $\beta(y)$ is $1$ and the period of $\beta'(y)$ is $0$, and hence in both cases $q = 01$ or $q = 10$ and then $d=1/2$. 

\medskip

We also observe that actually $d$ does not depend on the specific choice for the period $q$ in $\sigma(\beta(y))$ or $\sigma(\beta'(y))$, in that $d$ does not change if we choose a different starting digit for the period of the word or consider as a period any multiple of a minimal one. In fact, $d$ coincides with the asymptotic relative frequency $f(1,\sigma(\beta(y)))$ or $f(1,\sigma(\beta'(y)))$.

\begin{proposition}\label{dimratfib}
For every $y \in \Q$ we have $\dim_H R^{-1}(y) = -\log_2{t}$, with $t$ the unique real number in $(0,1)$ verifying the equation
\begin{equation*}
t^2+t^{1/d} = 1\,,
\end{equation*}
where $d = |q|_1/|q| \in (0,1]$ is the density of the $1$'s in the period $q$ of $\sigma(\beta(y))$ or $\sigma(\beta'(y))$.
\end{proposition}

\begin{proof}
According to \eqref{R-1bis} and \eqref{ratfibeq}, and recalling that $\dim_H(A \cup B) = \max\{\dim_HA,\dim_HB\}$ for every $A,B \subset [0,1]$, we only need to compute $\dim_H \xi(\fibexp{v\per{q}})$ with $v\per{q} \in \S$. Moreover, thanks to Lemma \ref{psibc} and the bi-Lipschitz invariance of the Hausdorff dimension, we can limit ourselves to consider the case when $v = \epsilon$, since $v\per{q} \approx \per{q}$ and both $v\per{q}$ and $\per{q}$ belong to $\S$. In addition, by applying once again Lemma \ref{psibc}, we can assume that the last digit of $q$ is 1 without changing the $\approx$ class. Indeed, $q$ must contain some 1 and thus it can be written as $q = q'1q''$ for some (possibly empty) binary words $q'$ and $q''$, whence $\per{q} = \pper{q'1q''} = q'1\pper{q''q'1} \approx \pper{q''q'1}$.

Given $q = q_1q_2\dots q_n \in \{0,1\}^n$ with $q_n = 1$, we put $m = |q|_1$ and consider the family of contractive similarities
\begin{equation*}
\T = \{T_a:\mathbb R \to \mathbb R\}_{a \in \mathbb N^m}
\end{equation*}
defined by
\begin{equation*}
T_{a}(x) = x/2^{2(a_1+ \ldots+a_m)+n} + x_a\,,
\end{equation*}
where $x_a = 0.\binexp{a}{q}$, for every $a \in \mathbb N^m$. Such family $\T$ is clearly relatively compact in the space of all the contractive similarities of $\mathbb R$ with the topology of the uniform convergence over bounded sets.

In terms of binary expressions, since $2(a_1 + \ldots + a_m)+n = |\binexp{a}{q}|$, we have 
\begin{equation}\label{Taexp}
T_a(0.x_1x_2 \dots) = 0.\binexp{a}{q}x_1x_2 \dots\,. 
\end{equation}

This allows us to derive in a straightforward way that
\begin{equation*}
\xi(\fibexp{\per{q}}) = \cup_{a \in \mathbb N^m} T_a(\xi(\fibexp{\per{q}})).
\end{equation*}
Indeed, the equality $\binexp{a}{q}\per{q} = \binexp{a\per{0}}{\per{q}}$ gives the inclusion $T_a(\xi(\fibexp{\per{q}})) \subset \xi(\fibexp{\per{q}})$ for every $a \in \mathbb N^m$. On the other hand, for every $x = \xi(\binexp{a}{\per{q}}) \in \xi(\fibexp{\per{q}})$ with $a \in \A$, we have $x = \xi(\binexp{a_1a_2\dots a_m}{q}\binexp{a_{m+1}a_{m+2}\dots}{\per{q}}) \in T_{a\trunc{m}}(\xi(\fibexp{\per{q}}))$.

From \eqref{Taexp}, we also see that $T_a((0,1)) \subset (0,1)$ for every $a \in \mathbb N^m$, and 
$T_a((0,1)) \cap T_{a'}((0,1)) = \emptyset$ for every $a,a'\in \mathbb N^m$ with $a \neq a'$, that is the family $\T$ satisfies the so called ``open set condition''.

Then, we are in position to apply Theorem 2.2 of \cite{Moran} (see also \cite[Theorem 3.11]{Fernau}), 
in order to conclude that
\begin{equation}\label{dimHeq}
\dim_H(\xi(\fibexp{b})) = \inf\{s > 0 \st \!\textstyle\sum_{a \in \mathbb N^m}r_a^s \leq 1\},
\end{equation}
where $r_a = 1/2^{2(a_1 + \ldots + a_m)+n}$ denotes the similarity ratio of $T_a$, for every $a \in \mathbb N^m$.

Now, recalling that $d = m/n$, we have
\begin{eqnarray*}
\textstyle\sum_{a \in \mathbb N^m}r_a^{\mkern3mus}
& \!\!\!=\!\!\! & 
\textstyle\sum_{a \in \mathbb N^m} 1/2^{(2(a_1 + \ldots + a_m)+n)s}\\[2pt]
& \!\!\!=\!\!\! &
\textstyle \big(\sum_{a_1 \geq 0}1/2^{2a_1s}\big) \cdots \big(\sum_{a_m \geq 0}1/2^{2a_ms}\big)/2^{ns}\\[2pt]
& \!\!\!=\!\!\! &
2^{-ns}/(1-1/2^{2s})^m = (2^{s/d} (1-2^{-2s}))^{-m}.
\end{eqnarray*}
Therefore, from \eqref{dimHeq} we deduce that $\dim_H(\xi(\fibexp{b}))$ is given by the unique $s > 0$ where the monotonic real function $f(s) = 2^{s/d}(1-2^{-2s})$ assumes the value 1. Finally, the equation $f(s)=1$ becomes $t^2 + t^{1/d} = 1$ as in the statement, by putting $s = -\log_2t$.
\end{proof}

We remark that actually the open set condition as formulated in the proof above also implies that the rational fibers of $R$ have positive Hausdorff measure in the relative Hausdorff dimension (see \cite[Theorem 2.2]{Moran}).

\medskip

To conclude this section, we want to provide a common lower and upper bound for the Hausdorff dimension of all the fiber of $R$. To this end, we first apply Proposition \ref{dimratfib} to compute the Hausdorff dimension of the two ``extremal'' rational fibers $R^{-1}(1)$ and $R^{-1}(1/3)$. We have
\begin{equation*}
\begin{array}{llll}
\beta(1) = \per{1} & \!\!\Rightarrow\, \sigma(\beta(1)) = \pper{01} 
& \!\!\Rightarrow\, d=1/2 & \!\!\Rightarrow\, \dim_H R^{-1}(1) = 1/2\,,\\[8pt]
\beta(1/3) = \pper{01} & \!\!\Rightarrow\, \sigma(\beta(1/3)) = \per{1}
& \!\!\Rightarrow\, d=1 & \!\!\Rightarrow\, \dim_H R^{-1}(1/3) = \log_2\varphi\,,
\end{array}
\end{equation*}
where $\varphi = (1+\sqrt5)/2$ is the golden ratio.

\begin{proposition}\label{dimfib}
For every $y \in [0,1]$, the following holds
\begin{equation*}
1/2 \leq \dim_H R^{-1}(y) \leq \log_2\varphi\,.
\end{equation*}
\end{proposition}

\begin{proof}
Thanks to \eqref{R-1bis} and the formula $\dim_H(A \cup B) = \max\{\dim_HA,\dim_HB\}$ holding for every $A,B \subset [0,1]$, it suffices to prove that $1/2 \leq \dim_H \xi(\fibexp{b}) \leq \log_2\varphi$ for every $b \in \S$. This immediately follows from Lemma \ref{psibc}, taking into account the obvious relation $\per{1} \succcurlyeq b \succcurlyeq \pper{01}$ and the fact that the Hausdorff dimension cannot increase under Lipschitz maps.
\end{proof}


\section{Dynamical properties}
\label{dynamics}

For the sake of convenience, we extend the notation $\rho_v$ introduced in \eqref{rhok} and \eqref{rhov}
by defining the maps
\begin{equation*}
\rho_v^n: \{0,1\}^\infty \to \{0,1\}^\infty
\end{equation*}
for every finite binary word $v$ and every $n \geq 1$ as follows. 

\pagebreak 

We start by setting
\begin{equation}\label{rho1}
\rho^1_v(w) = \rho_v(w)\,,
\end{equation}
with $w \in \{0,1\}^\infty$, and then for $n > 1$ we inductively define
\begin{equation}\label{rhoh}
\rho^n_v = \rho_{\rho^{n-1}(v)} \circ \rho^{n-1}_v\,.
\end{equation}

We emphasize that, in general $\rho^n_v$ does not coincide with $(\rho_v)^n$, the $n$-th iteration of the map $\rho_v$, when $v \neq \epsilon$ (for $v = \epsilon$ we have $\rho_\epsilon = \rho$, and hence $\rho_\epsilon^n = \rho^n = (\rho_\epsilon)^n$ for every $n \geq 1$). Indeed, by expanding the inductive definition we obtain
\begin{equation}\label{rhonv}
\rho_v^n = \rho_{\rho^{n-1}(v)} \circ \rho_{\rho^{n-2}(v)} \circ \ldots \circ \rho_{\rho(v)} 
\circ \rho_v\,,
\end{equation}
which is a composition of $n$ maps, each equal to $\rho$ or $\widetilde\rho$. Namely, the $i$-th map in the composition is $\rho$ or $\widetilde\rho$ depending on the parity of the length $|\rho^{n-i}(v)|$ of the image of $v$ under the $(n-i)$-th power of $\rho$. Nevertheless, a straightforward induction on $n \geq 1$ based on \eqref{rhov} gives the identity
\begin{equation}\label{rhonvw}
\rho^n(vw) = \rho^n(v)\rho^n_v(w)\,,
\end{equation}
where only the first two occurrences of $n$ denote iterations of maps.

\medskip

As an consequence of the surjectivity of $\rho$, and hence of $\widetilde\rho$, also all the maps $\rho_v^n$ are surjective. In fact, as it happens for $\rho$ and $\widetilde\rho$, the inverse image $(\rho_v^n)^{-1}(w)$ is countably infinite if $w$ is a finite non-empty word, while it is uncountable if $w$ is an infinite word, that is $w \in \B$.

On the other hand, contrary to what happens for $\rho$ and $\widetilde\rho$, the fact that $w \in \C$ is not enough to guarantee that $\rho_v^n(w) \in \B$ when $n > 1$. For example
\begin{equation*}
\rho_{110}^2(\pper{01}) = (\rho_{\rho(110)} \circ \rho_{110})(\pper{01}) =
(\rho_{01} \circ \rho_{110})(\pper{01}) = \rho(\widetilde\rho(\pper{01})) = \rho(\per{0}) = \epsilon\,.
\end{equation*}

\medskip

The next two lemmas constitute fundamental technical tools used for establishing the dynamical properties of the map $R$.

\begin{lemma}\label{lemmacyl}
For any sequence $(w_n)_{n \geq 0}$ of non-empty finite binary words and any sequence $(k_n)_{n \geq 0}$ of natural numbers with $k_0 = 0$ and $k_n \geq k_{n-1} + n_\epsilon(w_{n-1})$ for every $n \geq 1$, there is an uncountable set $B \subset \cyl{w_0} \cap \C \subset \B$ such that $\rho^{k_n}(b) \in \cyl{w_n}$ for every $b \in B$ and every $n \geq 1$.
\end{lemma}

\begin{proof}
For any sequences $(w_n)_{n \geq 0}$ and $(k_n)_{n \geq 0}$ as above, we construct the generic element of $B$ as an infinite concatenation $b = v_0v_1 \ldots \in \C$, where the $v_n$'s are non-empty finite binary words inductively defined as follows. We start with $v_0 = w_0$. Then, given $v_0,v_1, \dots, v_{n-1}$ with $n > 0$, we put $u_n = v_0v_1 \dots v_{n-1}$ and we choose $v_n$ to be any of the infinitely many (non-empty) words in $(\rho_{u_n}^{k_n})^{-1}(w_n)$ which ends with $1$. We observe that, $\rho^{k_n}(u_n) = \epsilon$ for every $n \geq 1$, whatever the choice of the $v_n$'s. In fact, $\rho^{k_1}(u_1) = \rho^{k_1}(v_0) = \epsilon$ since $k_1 \geq n_\epsilon(v_0)$, while for $n >1$, taking into account that $k_n - k_{n-1} \geq n_\epsilon(w_{n-1})$, we have by induction
\begin{equation*}
\begin{array}{rcl}
\rho^{k_n}(u_n) 
& \!\!\!=\!\!\! & \rho^{k_n-k_{n-1}}(\rho^{k_{n-1}}(u_{n-1}v_{n-1}))\\[2pt]
& \!\!\!=\!\!\! & \rho^{k_n-k_{n-1}}(\rho^{k_{n-1}}(u_{n-1})\rho^{k_{n-1}}_{u_{n-1}}(v_{n-1}))\\[2pt]
& \!\!\!=\!\!\! & \rho^{k_n-k_{n-1}}(w_{n-1})) = \epsilon\,.
\end{array}
\end{equation*}
The condition that each $v_n$ ends with 1 prevents different choices for the $v_n$'s to produce the same final word $b$, as it can be easily deduced from the fact that $|v_n|_1 = |w_n|$ does not depend on the specific choice of $v_n$. In this way, as the result of infinitely many infinite choices, one for each $n > 0$, we get an uncountable set $B$ of words $b = v_0v_1 \ldots \in \cyl{w_0} \cap \C$ such that
\begin{equation*}
\rho^{k_n}(b) = \rho^{k_n}(u_n v_n v_{n+1} \dots) = \rho^{k_n}(u_n) \rho^{k_n}_{u_n}(v_n)
\rho_{u_nv_n}^{k_n}(v_{n+1}\dots) = w_n \rho_{u_nv_n}^{k_n}(v_{n+1}\dots) \in \cyl{w_n}
\end{equation*}
for every $n \geq 1$, where the second equality can be obtained by two applications of \eqref{rhonvw}, while the third one immediately follows from $\rho^{k_n}(u_n) = \epsilon$ and $\rho^{k_n}_{u_n}(v_n)=w_n$.
\end{proof}

\begin{lemma}\label{lemmacylbis}
For any sequence $(w_n)_{n \geq 0}$ of non-empty finite binary words there is an uncountable set $B \subset \cyl{w_0} \cap \C \subset \B$ with the property that for every $b \in B$ and every $n,m \geq 0$, there exists $k \geq m$ such that $\rho^k(b) \in \cyl{w_n} \subset \B$.
\end{lemma}

\begin{proof}
Possibly by replacing $(w_n)_{n \geq 0}$ with the concatenation of all its finite initial subsequences, we can assume that each word $w_n$ appears infinitely many times in the sequence.
Under this assumptions, it is enough to prove the existence of an uncountable subset $B \subset \cyl{w_0} \cap \C$ with the weaker property that for every $b \in B$ and $n \geq 1$ there is $k \geq 1$ such that $\rho^k(b) \in \cyl{w_n}$. The existence of such a set $B$ is guaranteed by the previous lemma. 
\end{proof}

We now proceed discussing some asymptotic properties of the orbits of the map $R$. First of all, we show that the set of rationals $\Q = \mathbb Q \cap[0,1]$ is $R$-invariant and establish the asymptotic behaviour of the restriction of $R\restr{\Q}: \Q \to \Q$.

\begin{proposition}\label{rationals}
The set $\Q \subset [0,1]$ is $R$-invariant. Moreover, $R\restr{\Q}$ admits only two periodic orbits, namely $C_0 = \{0,2/3\}$ and $C_1 = \{1,1/3\}$, and $\Q$ decomposes as the disjoint union of two dense subsets
\begin{equation}\label{Qi}
\Q = \Q_0 \cup \Q_1\,,
\end{equation}
where $\Q_i$ consists of all the rationals in $\Q$ whose forward orbit contains $C_i$, for $i=1,2$.
\end{proposition} 

\begin{proof} 
The fact that $C_0 = \{0,2/3\}$ or $C_1 = \{1,1/3\}$ are 2-cycles can be trivially verified.\break Then, the $R$-invariance of $\Q$ follows once we prove that $R(x) \in \Q$ for every
$x \in \Q \setminus \{0,1\}$.

Given any $x \in \Q \setminus \{0,1\}$, we can write $\beta(x) = w\per{p} \in \C$ with $w$ and $p$ finite binary words of minimal length such that $|p|_1 > 0$. Then, by applying $\rho$ we obtain
\begin{equation*}
\rho(\beta(x)) = \rho(w\per{p}) = \rho(w) \rho_{w}(p) \rho_{wp}(p) \rho_{wp^2}(p) \dots
\end{equation*}
If $|p|$ is even, $|wp^k|$ has the same parity of $|w|$, and hence $\rho_{wp^k} = \rho_w$ for every $k \geq 1$. If instead $|p|$ is odd, $|wp^k|$ has the same parity of $|w| + k$, and so $\rho_{wp^k}$ coincides with $\rho_w$ for $k$ even while it coincides with the complementary map $\widetilde\rho_{w}$ for $k$ odd. Thus, we can rewrite $\rho(\beta(x))$ as
\begin{equation*}
\rho(\beta(x)) = 
\left\{\vrule width0pt height16pt\right.\!\!\!
\begin{array}{ll}
\rho(w)\pper{\rho_w(p)} & \text{if $|p|$ is even}\\[2pt]
\rho(w)\pper{\rho_w(p)\widetilde\rho_{w}(p)} & \text{if $|p|$ is odd}
\end{array}.
\end{equation*}
In both cases $\rho(\beta(x))$ is periodic, which implies that $R(x) = \xi(\rho(\beta(x)) \in \Q$.
This concludes the proof of the $R$-invariance of $\Q$.

Now, we pass to prove that the forward orbit of any $x \in \Q \setminus \{0,1\}$ contains one of $C_0$ or $C_1$.\break
From the last formula we derive that $\rho(\beta(x)) \in \C$, that is $\rho(\beta(x)) = \beta(R(x))$, except when $|p|$ is even and $|\rho_w(p)|_1 = 0$, in which case $\rho(\beta(x)) = \beta'(R(x)) = \rho(w)\per{0}$. Here, we have two possibilities, either $|\rho(w)|_1 = 0$ and then $R(x) = 0$ or else $|\rho(w)|_1 > 0$ and then $\beta(R(x)) = v\per{1}$ with $n_\epsilon(v) \leq n_\epsilon(\rho(w))$ (in fact $v = \rho(w)\trunc{k-1}0$, where $k$ the position of the last 1 in $\rho(w)$).\break
By iterating the process, we can conclude that $R^n(x)$ admits a purely periodic binary expansion for every $n \geq n_\epsilon(w)$.

On the other hand, for $x' \in \Q \setminus \{0,1\}$ such that $\beta(x') = \per{q}$ with $q$ a finite binary word of minimal length such that $0 < |q|_1 < |q|$, we have
\begin{equation*}
\rho(\beta(x')) = 
\left\{\vrule width0pt height16pt\right.\!\!\!
\begin{array}{ll}
\pper{\rho(q)} & \text{if $|q|$ is even}\\[2pt]
\pper{\rho(q)\widetilde\rho(q)} & \text{if $|q|$ is odd}
\end{array}.
\end{equation*}
The minimality of $q$, implies that $|\rho(q)| = |q|_1 < |q|$ if $|q|$ is even, while $|\rho(q)\widetilde\rho(q)| = 2|\rho(q)| = 2 |q|_1 < 2 |q|$ if $|q|$ is odd. In this last case, $|\rho(q)\widetilde\rho(q)|$ is even, $|\rho(q)\widetilde\rho(q)|_1 = |\rho(q)\widetilde\rho(q)|/2$, and $\rho(\beta(x')) = \beta(R(x'))$. Therefore, either $R(x') = \xi(\rho(\beta(x')))$ or $R^2(x')= \xi(\rho^2(\beta(x')))$ admits a purely periodic binary expansion whose period length is strictly less than $|q|$. By iteration, we eventually get $R^n(x') = 0$ or $R^n(x') = 1$ for a sufficiently large $n$.

\pagebreak 

At this point, we are left to show that $\Q_0$ and $\Q_1$ are dense, or equivalently that they meet all the dyadic intervals $\xi(\cyl{w}) \subset [0,1]$ with $w$ any finite binary word. 

First, let us argue for $Q_0$. We put $n = n_\epsilon(w)$, and define a sequence of finite binary words $p_0, p_1, \dots, p_n$ by induction on $i$ decreasing from $n$ to $0$, as follows. We start with $p_n = 0$, and given $p_i$ with $0 < i \leq n$, we let $p_{i-1}$ be any even length element of $(\rho_{\rho^{i-1}(w)})^{-1}(p_i)$. The existence of such an element is guaranteed by the surjectivity of $\rho_{\rho^{i-1}(w)}$ and the possibility of appending a 0, if needed to make the length even, without changing the image under $\rho_{\rho^{i-1}(w)}$. 

Then, we put $x = \xi(w\per{p_0\!}) \in \xi(\cyl{w})$. By induction on $i$, we get
\begin{equation*}
R^i(x) = \xi(\rho^i(w)\per{p_i})
\end{equation*}
for every $i = 0,\dots,n$, as follows. The base of the induction is the trivial case of $i = 0$, while the inductive step is given by
\begin{equation*}
\begin{array}{rcll}
R^i(x) & \!\!\!=\!\!\! & 
R(\xi(\rho^{i-1}(w)\per{p_{i-1}})) & \quad (\text{the inductive hypothesis})\\[2pt]
& \!\!\!=\!\!\! & 
\xi(\rho(\rho^{i-1}(w)\per{p_{i-1}})) & \quad (\text{since $\rho^{i-1}(w)\per{p_{i-1}} \in \C$})\\[2pt]
& \!\!\!=\!\!\! & 
\xi(\rho^i(w)\rho_{\rho^{i-1}(w)}(\per{p_{i-1}})) & \quad (\text{thanks to equation \eqref{rhoh}})\\[2pt]
& \!\!\!=\!\!\! & 
\xi(\rho^i(w)\pper{\rho_{\rho^{i-1}(w)}p_{i-1}}) & \quad (\text{since $p_{i-1}$ has even length})\\[2pt]
& \!\!\!=\!\!\! & 
\xi(\rho^i(w)\per{p_i}) & \quad (\text{by definition of $p_{i-1}$}).\\[2pt]
\end{array}
\end{equation*}
In particular, we have $R^n(x) = \xi(\rho^n(w)\per{p_n}) = \xi(\per{0}) = 0$, as desired.

This concludes the proof of the density of $\Q_0$. The same argument, but starting from $p_n = 1$ instead of $p_n = 0$, proves that $\Q_1$ is dense.
\end{proof}

We remark that $C_0$ and $C_1$ are the only periodic $R$-orbits in $\Q$ and that $\Q$ does not contain any dense $R$-orbit. Thus, in the following, when considering any other periodic $R$-orbit or any dense $R$-orbit, we can always assume that they are disjoint from $\Q$, avoiding in this way the technicalities due to the double binary expansion of the dyadic rationals.

\begin{proposition}\label{periodic_points}
The set of periodic points of $R$ is dense in $[0,1]$ and contains uncountably many $n$-periodic points for any given minimal period $n \geq 1$.
\end{proposition} 

\begin{proof}
In the light of the above observation, apart from the 2-cycles $C_0$ and $C_1$, all the other $n$-periodic $R$-orbits are contained in $[0,1] \setminus \Q$ and then they bijectively correspond to those of $\rho$ through the unique binary expansion $\beta(x)$ of every $x \in [0,1] \setminus \Q$.

Therefore, the first part of the statement follows once we prove that for any non-empty finite binary word $w$ there are uncountably many infinite binary $\rho$-periodic words $b$ in $\cyl{w} \cap \C$, whose period is the vanishing order $n_\epsilon(w)$. In fact, it is enough to observe that the corresponding points $\xi(b)$ in the generic dyadic interval $\xi(\cyl{w}) \subset [0,1]$ are $R$-periodic with the same period.

Fixed any non-empty finite binary word $w$, we put $n = n_\epsilon(w) \geq 1$ and consider all the infinite concatenations
\begin{equation*}
b = w_0w_1w_2 \dots w_k \ldots \in \C\,,
\end{equation*}
where $w_0 = w$ and for every integer $k \geq 1$ we let $w_k$ be any of the infinitely many (non-empty) words in $(\rho^n_{w_0w_1\dots w_{k-1}})^{-1}(w_{k-1})$ which ends with 1. As in the proof of Lemma \ref{lemmacyl}, this last condition guarantees that there are no duplicates in the resulting words $b$, which are thus uncountably many.
Then, we have
\begin{eqnarray*}
\rho^n(b) 
& \!\!\!=\!\!\! &
\rho^n(w_0w_1w_2 \dots w_k \dots)\\[2pt]
& \!\!\!=\!\!\! &
\rho^n(w_0)\rho^n_{w_0}(w_1)\rho^n_{w_0w_1}(w_2) \dots \rho^n_{w_0w_1\dots w_{k-1}}(w_k) \dots\\[2pt]
& \!\!\!=\!\!\! &
\epsilon w_0 w_1 \dots w_{k-1} \ldots = b\,,
\end{eqnarray*}
and hence $b$ has period $n$.

\pagebreak 

For the second part we need the additional fact that for every $n \geq 1$ is the vanishing order of all the non-empty finite binary words $w$ in $\rho^{-n}(\epsilon)$, for example $w = \sigma^{n-1}(0)$, and that such a word $w$ can be chosen so that no $\rho^k(w)$ is a prefix of $w$ for $k < n$, for example $w = 00\sigma^{n-1}(0)$. This last property immediately implies that $\rho^k(b) \neq b$ for every $b$ as above and $k = 1, \dots, n-1$, that is $n$ is the minimal period of $b$.
\end{proof}

In particular, Proposition \ref{periodic_points} tells us that the set $\Fix R$ of fixed points of $R$ is uncountable, meaning that the graph of $R$ crosses the diagonal uncountably many times.
In fact, by taking $w = 0^\ell$ in the previous proof, the resulting word $b = w_0w_1 w_2 \ldots \in \C$ belongs to $\Fix\rho$, and hence $\xi(b) \in \Fix R$. In the special case when $w = w_0 = 0$
and each $w_k$ is chosen to have minimal length among the words with the prescribed properties, that is either $w_k = \sigma(w_{k-1})$ or $w_k = \sigma(\widetilde w_{k-1})$ according to the parity of $|w_0w_1\dots w_{k-1}|$, we get as $b$ the word
\begin{equation}\label{b0}
\textstyle b^0 = 00101\prod_{h=0}^{\infty}(1^{3\cdot 2^h}(01)^{3\cdot 2^h}) = 00101111010101111111010101010101 \dots\,.
\end{equation}
The corresponding real number
\begin{equation}\label{x0}
x^0 = \xi(b^0) = 0.00101111010101111111010101010101 \dots
\end{equation}
turns out to be the largest element in $\Fix R$, as we will see in a while.

\medskip

Now, we want to make more explicit the procedure given in the proof of Proposition \ref{periodic_points} in the specific case of fixed points. To any sequence $a = (a_1, a_2, \dots) \in \A = \mathbb N^\omega$, we associates a (distinct) fixed point $\fixexp{a} \in \Fix\rho \subset \C$ and a corresponding (distinct) fixed point $x^a = \xi(\fixexp{a}) = 0.b^a \in \Fix R \subset [0,1]$, with $b^0$ and $x^0$ associated to the null sequence.

\medskip

First, we observe that a non-empty finite binary word $w$ can be completed to a fixed point $wb \in \Fix \rho$ for a suitable $b \in \B$ if and only if $\rho(w)$ is a prefix of it. Moreover, if this is the case then there is a unique word $b \in \S$ such that $b^w = wb \in \Fix\rho$. The word $b^w$ can be constructed, by generalizing the construction of $b^0$ outlined above. Since $\rho(w)$ is a prefix of $w$, we can write $w = \rho(w)v_0$ for a certain word $v_0$, which can be easily seen to be non-empty. Then, we define $b^w$ as the infinite concatenation
\begin{equation*}
b^w = wv_1v_2 \ldots \in \Fix\rho\,,
\end{equation*}
where the words $v_k$ for $k \geq 1$ are inductively defined by starting from $v_0$ and putting $v_k = \sigma(v_{k-1})$ or $v_k = \sigma(\widetilde v_{k-1})$ if $|wv_1 \dots v_{k-1}|$ is even or odd, respectively (compare with the construction of $b^0$ above). Then $b = v_1v_2\ldots \in \S$, since each $v_k$ contains no pair of consecutive 0's and ends with 1. Moreover, these properties of $v_k$ make it uniquely determined by $v_{k-1}$ for every $k \geq 1$, which implies the uniqueness of $b$ in $\S$ by induction.

Now, given any $a = (a_1, a_2, \dots) \in \A$, we put $\fixexp{a}_0 = b^0$ and define the sequence $(\fixexp{a}{\!\!}_{n})_{n\geq 0} \subset \Fix\rho$ as follows. For every $n \geq 1$, let $k_n$ be the position of the $n$-th occurrence of 1 in $\fixexp{a}{\!\!}_{n-1}$, and $w^a_n$ be the word obtained from $\fixexp{a}{\!\!}_{n-1}\trunc{k_n}$ by inserting the word $0^{2a_n}$ before the last digit 1. The word $w^a_n$ share with $\fixexp{a}{\!\!}_{n-1}\trunc{k_n}$ the property that its image $\rho(w^a_n)=\rho(\fixexp{a}{\!\!}_{n-1}\trunc{k_n})$ is a prefix of it, so we can set $\fixexp{a}{\!\!}_n = b^{w^a_n}$. Notice that $|w^a_n| = k_n + 2a_n \geq k_n > k_{n-1}+2a_{n-1} = |w^a_{n-1}|$ for every $n > 1$.
Since the sequence $(k_n)_{n\geq 1}$ is increasing and $\fixexp{a}{\!\!}_{n}\trunc{k_n-1} = \fixexp{a}{\!\!}_{n-1}\trunc{k_n-1}$ for every $n \geq 1$, there exists the limit
\begin{equation}\label{fixexp}
\fixexp{a} = \lim_{n \to \infty} \fixexp{a}{\!\!}_n\,,
\end{equation}
which is a fixed point for $R$, admitting by construction all the prefixes $w^a_n$ for $n \geq 1$.

\medskip

In order to characterize the structure of $\Fix R$, we need to prove that the previous construction produces all the points in $\Fix\rho$. Actually, we prove something more in the following Lemma.

\begin{lemma}\label{phi}
The map $\phi: \A \to \Fix\rho$ given by $a \mapsto \phi(a) = \fixexp{a}$ is a Lipschitz homeomorphism.
\end{lemma}

\begin{proof}
For any $b \in \Fix\rho$, we consider the sequence $a^b= (a^b_1,a^b_2 \dots) \in \A$, where $a^b_n$ is the number of (disjoint) pairs of $0^s$ immediately preceding the $n$-th occurrence of 1 in $b$. Then, the map $\psi:\Fix\rho \to \A$ given by $b \mapsto a^b$ is the inverse of $\phi$, as straightforward verification shows.

Now, for any $a,a' \in \A$ we have that $d(a,a') = 1/2^{n+1}$ if and only if $a\trunc{n} = a'\trunc{n}$ and $a_{n+1} \neq a'_{n+1}$, or equivalently, with the above notations, $w^a_n = w^{a'}_n$ and $w^a_{n+1} \neq w^{a'}_{n+1}$. Therefore, $\fixexp{a}$ and $\fixexp{a'}$ share the same prefix of length $|w^a_n| = k_n + 2a_n \geq k_n \geq n$, while they have different prefixes of length $|w^a_{n+1}| = k_{n+1} + 2a_{n+1}$.
So, we have $1/2^{k_{n+1}+2a_{n+1}} \leq d(\phi(a),\phi(a')) \leq 1/2^{n+1}$, which means that $\phi$ is a Lipschitz map (second inequality) and that $\psi$ is continuous at $\phi(a)$ (first inequality).
\end{proof}

\begin{proposition}\label{fixcantor} 
The fixed point set $\Fix R$ is a uncountable subset of $[0,1] \setminus \D$ and its closure $\Cl(\Fix R)$ is a null measure Cantor set in $[0,1]$ with a countable remainder $\Cl(\Fix R) \setminus \Fix R \subset \D$.
\end{proposition}

\begin{proof}
The inclusion $\Fix R \subset [0,1] \setminus \D$ follows by Proposition \ref{rationals}. As a consequence, the maps $\xi$ and $\beta$ restrict to a bijective maps $\Fix \rho \leftrightarrow \Fix R$, and hence $\Fix R$ is uncountable.

We now prove that $\Cl(\Fix R) \subset \Fix R \cup \D$, which trivially gives $\Cl(\Fix R) \setminus \Fix R \subset \D$. The argument is similar to the first part of the proof of Proposition \ref{fibercantor}.
Let $x \in \Cl(\Fix R)$ and $(x_n)_{n \geq 1} \subset \Fix R$ be a sequence such that $\lim_{n \to \infty} x_n = x$. Then, by Lemma \ref{phi} there is a sequence $(a_n)_{n \geq 1} \subset \A$ such that $x_n = \xi(\phi(a_n))$ for every $n \geq 1$. If $x \in \D$ we are done, otherwise the continuity of $\beta$ at $x$ implies that $\lim_{n \to \infty} \phi(a_n) = \beta(x) \in \C$. If the sequence of the $i$-th components $((a_n)_i)_{n \geq 1}$ is bounded for every $i \geq 1$, then by compactness we can replace the sequence $(a_n)_{n \geq 1}$ by a subsequence converging to $a \in \A$, and by the continuity of $\phi$ we have $\beta(x) = \phi(a)$, that is $x = \xi(\phi(a)) \in \Fix R$. If instead some sequence $((a_n)_i)_{n \geq 1}$ is unbounded, let $m$ be minimum index $i \geq 1$ for which this happens. In this case, we replace the sequence $(a_n)_{n \geq 1}$ by a suitable subsequence such that the sequence of truncations $((a_n)\trunc{m-1})_{n \geq 1}$ is constant, and hence $\phi(a_n)\trunc{k_{m-1}}$, where $k_{m-1}$ is the index of the $(m-1)$-th occurrence of 1, coincides with the same finite word $d$ for every $n \geq 1$. Then, we get $\beta(x) = \lim_{n \to \infty}\phi(a_n) = d\mkern1mu\per{0}$, which is absurd since $\beta(x) \in \C$. This concludes the proof of the inclusion $\Cl(\Fix R) \subset \Fix R \cup \D$.

Concerning the null measure property, let us consider the map
\begin{equation*}
\phi \circ \phi_{\per{1}}^{-1}:\fibexp{\per{1}} \to \Fix \rho\,,
\end{equation*}
where $\phi_{\per{1}}: \A \to \fibexp{\per{1}}$ and $\phi: \A \to \Fix\rho$ are the homeomorphisms defined in the proof of Lemma \ref{phibc} and in Lemma \ref{phi}, respectively. This map sends $\binexp{a}{\per{1}} \!\! \in \fibexp{\per{1}}$ to $\fixexp{a} \in \Fix\rho$ for every $a \in \A$, and it can be shown to be a 1-Lipschitz map as follows. Given any $a,a' \in \A$ such that $a\trunc{n} = a'\trunc{n}$ and $a_{n+1} \neq a'_{n+1}$ we put $h(a,a') = a_1 + \ldots + a_n + \min\{a_{n+1},a'_{n+1}\}$, as in the proof of Lemma \ref{phibc}. Then, the maximum index $i$ for which $\binexp{a}{\per{1}}\!\!\!\trunc{i} = \binexp{a}{\per{1}}\!\!\!\trunc{i}$ is given by $2h(a,a') + n$, while the maximum index $i$ for which $\fixexp{a}\trunc{i} = \fixexp{a'}\trunc{i}$ is given by $2h(a,a') + k_n$, where $k_n$ is the index of the $n$-th occurrence of 1 in the word $\fixexp{a}_{n-1}$ defined in the construction of $\fixexp{a}$ just before Lemma \ref{phi}. Since $k_n \geq n$, we can conclude that $\phi \circ \phi_{\per{1}}^{-1}$ is a 1-Lipschitz map. From this, arguing as in the proof of Lemma \ref{psibc}, we can derive the Lipschitz condition for the map
\begin{equation*}
\xi \circ \phi \circ \phi_{\per{1}}^{-1} \circ \beta:\xi(\fibexp{\per{1}}) \to \Fix R\,.
\end{equation*}
Therefore, thanks to Proposition \ref{dimfib} and its proof, we have $\dim_H \Fix R \leq \dim_H \xi(\fibexp{\per{1}}) = \log_2 \varphi < 1$, which implies that $\Fix R$ has a null Lebesgue measure.
Since the remainder $\Cl(\Fix R) \setminus \Fix R$ is countable, also $\Cl(\Fix R)$ has a null Lebesgue measure.

As a consequence, $\Cl(\Fix R)$ is topologically 0-dimensional. On the other hand, it is a perfect set, because for every $a \in \A$, we have $\xi(\fixexp{a}) = \lim_{n \to \infty} \xi(\fixexp{a_n})$ with $a_n \in \A$ defined by $(a_n)_i = a_i + \delta_{i,n}$, for any $i,n \geq 1$.
So, $\Cl(\Fix R)$ is a Cantor set (see \cite{Kura}).
\end{proof}

We observe that the proof of Proposition \ref{fixcantor} is essentially based on the parametrization of $\Fix\rho$ given by Lemma \ref{phi}, which is nothing else than an explicit reformulation in the case of fixed points of the general construction provided in the proof of Proposition \ref{periodic_points} of periodic points of any period. So, we expect that similar results hold for periodic points of any given period.

\medskip

The idea is that, in order to generate all the periodic orbits of a given period $n > 1$, one could start from a certain finite set of ``simplest'' ones, and than progressively enlarge this set
by iterating in a suitable order the following procedure. Take any such $n$-orbit $(b_1,b_2, \dots, b_n)$, insert some pairs of consecutive 0's in any position of any $b_i$, and then reconstruct all the orbit accordingly.

\medskip

For example, starting from the 2-cycle $(\per{1},\per{(01)})$ of $\rho$, corresponding to the 2-cycle $C_1$ of $R$ (see Proposition \ref{rationals}), one can obtain in this way the Thue-Morse 2-cycle 
$(u,\widetilde u)$ of $\rho$ described in the next remark, and hence a corresponding 2-cycle for $R$.

\begin{remark}
Consider the well known Thue-Morse substitution rule on $\B$
\begin{equation*}
\widehat\tau:\left\{\vrule width0pt height12pt\right.\!\!\!
\begin{array}{ll}
0 \mapsto 01\\
1 \mapsto 10
\end{array},
\end{equation*}
and the digit-wise generated map $\widehat\tau:\B \to \B$. The map $\widehat\tau$ has two fixed points, the Thue-Morse sequence starting with $0$
\begin{equation*}
u = 0110 1001 1001 0110 1001 0110 0110 1001 \dots
\end{equation*}
and its complementary sequence $\widetilde u$, the Thue-Morse sequence starting with $1$, which was used in {\rm \cite{Thue}} to found the combinatorics of words. Recalling the definition of $\,\overline\tau$ in \eqref{tauinv} and noting that
\begin{equation*}
\overline\tau^2(0) = \widehat\tau^2(0) = 0110 
\quad \text{and} \quad 
\overline\tau^2(1) = \widehat\tau^2(1) = 1001\,,
\end{equation*}
we can easily verify that $(u,\widetilde u)$ is a 2-cycle of $\rho$. Then, letting $t = \xi(u)$ be the Thue-Morse number, we have the corresponding 2-cycle $(t,1-t)$ of $R$.
\end{remark}

Concerning the periodic points of $R$ having higher order $n$, in analogy with the description of the maximal fixed point $x^0$, we limit ourselves to consider the periodic points whose binary expansion is the ``simplest'' infinite binary word among those generated by the construction described in the proof of Proposition \ref{periodic_points}. Namely, the binary word obtained by starting from the word $w_0 = 0$ and choosing each $w_k$ to be the shortest word in $(\rho^n_{w_0 w_1 \dots w_{k-1}})^{-1}(w_{k-1})$, that is the image of $w_{k-1}$ under a suitable composition of $\sigma$'s and $\widetilde \sigma$'s.

\medskip

The case of $n$ even is not so interesting, since we always obtain $0.\per{1}$, while the binary expression of the ``simplest'' periodic point of odd order $n = 2\ell + 1 \geq 3$ is given by
\begin{equation}\label{xell}
\begin{array}{ll}
x^\ell = 0.(01)^{2^{\ell-1}}\!\prod_{h=0}^{\infty} 
&\!\!\!\!\!
\big(1^{2^{(2\ell+1)h+\ell}} (01)^{2^{(2\ell+1)h+\ell-1}} 1^{(2^\ell-1)\,2^{(2\ell+1)h+\ell}}\\[2pt]
&\!\!\!\!\! \ \ 
(01)^{2^{(2\ell+1)h+2\ell}}1^{2^{(2\ell+1)h+2\ell}}(01)^{(2^\ell-1)\,2^{(2\ell+1)h+2\ell}}\big)\,.
\end{array}
\end{equation}

In order to see that this point has order $n=2\ell+1$, we write it as $0.w\prod_{h=0}^{\infty}(u_hv_h)$, with
\begin{equation}\label{psw}
\begin{array}{rl}
w = (01)^{2^{\ell-1}}, 
&\!\! u_h = 1^{2^{(2\ell+1)h+\ell}} 
(01)^{2^{(2\ell+1)h+\ell-1}}1^{(2^{\ell}-1)\,2^{(2\ell+1)h+\ell}},\\[4pt]
&\!\! v_h = (01)^{2^{(2\ell+1)h+2\ell}}1^{2^{(2\ell+1)h+2\ell}}(01)^{(2^{\ell}-1)\,2^{(2\ell+1)h+2\ell}}.
\end{array}
\end{equation}

Then, we first note that the word $w$ vanish to the empty word $\epsilon$ in exactly $n$ iterates of $\rho$. Moreover, by a somewhat tedious but straightforward calculation, repeatedly using the rules
\begin{equation*}
\rho:\left\{\vrule width0pt height15pt\right.\!\!\!
\begin{array}{ll}
(01)^k \mapsto 1^k\\[4pt]
1^{2k} \mapsto (01)^k
\end{array},
\end{equation*}
for every $k \geq 1$, one checks that
\begin{equation*}
\rho^n_w(u_0) = \widetilde\rho\,^2 \circ \rho^{2\ell-1}(u_0) = w\,,
\end{equation*}
while
\begin{equation*}
\rho^n_{wu_0v_0 \dots u_h}(v_h) = \rho^{2\ell+1}(v_h) = u_h
\quad \text{and} \quad
\rho^n_{wu_0v_0 \dots v_h}(u_{h+1}) = \rho^{2\ell+1}(u_{h+1}) = v_h
\end{equation*}
for every $h \geq 0$. Putting all those equation together, one can conclude that $w\prod_{h=0}^{\infty}(u_hv_h)$ is a fixed point for $\rho^n$, and hence that $0.w\prod_{h=0}^{\infty}(u_hv_h)$ is a fixed point for $R^n$.

\medskip

The rest of this section is devoted to analyse some chaotic properties of the map $R$.

\begin{proposition}\label{dense_orbit} 
The set of points with a dense $R$-orbit is uncountable and dense in $[0,1]$.
\end{proposition} 

\begin{proof}
To get the existence of uncountably many points with a dense $R$-orbit, it suffices to take the sequence $(w_n)_{n\geq 0}$ in Lemma \ref{lemmacylbis} with the property that it contains all the non-empty finite binary words. The density of the points with a dense $R$-orbit follows immediately observing that we can take any non-empty finite word as $w_0$.
\end{proof}

\begin{proposition}\label{sensi}
For every interval $I \subset [0,1]$, there exists $n \geq 1$ such that $R^n(I) = [0,1]$.\break 
In particular, $R$ is topologically $($strongly\/$)$ mixing and has sensitive dependence on initial conditions.
\end{proposition}

\begin{proof}
Let $w$ be any non-empty finite binary word such that $\xi(\cyl{w}) \subset I$, and set $n = n_\epsilon(w)$. Pick $y\in [0,1]$. By the surjectivity of $\rho_w^n$, there is $v \in \C$ such that $\rho_w^n(v) = \beta(y)$, and hence $\rho^n(wv) = \rho^n(w)\rho^n_w(v) = \beta(y)$ as well. Then, for $x = \xi(wv)$ we have $x \in I$ and $R^n(x)=y$.

This implies sensitivity to initial conditions with sensitivity constant $1/2$. Since every open set is a countable union of open intervals, it also follows that, for every open sets $A, B\subset [0,1]$, there is $n \geq 1$ such that $R^m (A)\cap B \neq \emptyset$ for every $m>n$, that is $R$ is topologically mixing.
\end{proof}

\begin{corollary}
The map $R$ is chaotic in the sense of Devaney $($see e.g. \kern-2pt{\rm\cite{Devaney}}$)$. 
\end{corollary}

\begin{proof}
It follows by Propositions \ref{periodic_points}, \ref{dense_orbit} and \ref{sensi}. 
\end{proof}

Concerning distributional chaos, a concept introduced within the context of topological dynamical systems in 1994 (see \cite{Smital}), the map $R$ satisfies the strongest form of distributional chaos, namely the distributional chaos of type 1, DC1-chaos in short \cite{Smital2}. Actually, we prove that the DC1-chaos exhibited by $R$ is uniform in the sense of \cite{Oprocha}, as specified in the next proposition.

\begin{proposition}\label{DC1chaos}
The map $R$ is uniformly DC1-chaotic, meaning that it admits an uncountable uniformly DC1-scrambled set, that is an uncountable subset $S \subset [0,1]$ such that for every $\delta > 0$ and some $\overline\delta >0$ the following two conditions hold, where $\#$ denotes the cardinality,
\begin{eqnarray}
\label{DC11}
& \limsup_{n\to\infty}
\displaystyle\frac{1}{n}\,\#\{k = 0, \dots, n-1 \text{ such that } |R^k(x)-R^k(y)| < \delta\} = 1\,,
\\[2pt]
\label{DC12}
& \liminf_{n\to\infty}
\displaystyle\frac{1}{n}\, \#\{k = 0, \dots, n-1 \text{ such that } |R^k(x)-R^k(y)| < \overline \delta\} = 0\,,
\end{eqnarray}
for every $x \neq y$ in $S$.

\end{proposition}

\begin{proof}
We will explicitly construct a set $S =\{x_\alpha\}_{\alpha \in A}$ satisfying the properties \eqref{DC11} and \eqref{DC12}.\break
As the set of indices $A$ we take the image $A = \mu(\B) \subset \B$ of the map $\mu:\B \to \B$ defined by
\begin{equation*}
\mu(b)_n = \left\{\vrule width0pt height15pt\right.\!\!\!
\begin{array}{ll}
b_k & \text{if $n = p^k$ with $p$ prime and $k \geq 1$}\\[2pt]
0 & \text{otherwise}
\end{array},
\end{equation*}
for every $b \in \B$. We observe that $A$ is an uncountable set, being $\mu$ injective, and that all the elements $\alpha \in A$ share the same $n$-th component $\alpha_n = 0$ for infinitely many $n \geq 1$. Furthermore, since each $\alpha = \mu(b)$ contains infinite copies of $b$ as subsequences, any two different elements $\alpha,\alpha' \in A$ cannot be ultimately coinciding, that is for every $m \geq 1$ there is $n \geq m$ with $\alpha_n \neq \alpha'_n$.

Now, let $(e_n)_{n \geq 1}$ be the increasing sequence of natural numbers inductively defined by
\begin{equation}\label{expseq}
e_1 = 1 \quad \text{and} \quad e_n = n^2 \textstyle \sum_{k=1}^{n-1}e_k \ \text{for $n > 1$}\,,
\end{equation}
and let $(v_n)_{n \geq 1}$ and $(u_n)_{n \geq 1}$ be the two sequences of binary words inductively defined by
\begin{equation}\label{vnun}
\begin{array}{llll}
v_1 = 1 & \text{and} & v_n = \sigma^{2e_n}(v_{n-1}) & \text{for $n > 1$}\,,\\[2pt]
u_1 = 001 & \text{and} & u_n = 00\,\sigma^{2e_n}(u_{n-1}) & \text{for $n > 1$}\,.
\end{array}
\end{equation}
Clearly, we have $n_\epsilon(v_n) = n_\epsilon(u_n) = \sum_{k = 1}^n 2e_k$ for every $n \geq 1$. Moreover, the equality $\sigma^2(1^k) = 1^{2k}$ implies that 
\begin{equation}\label{manyones}
v_n = 1^{2^{\sum_{k=2}^{n}e_k}},
\end{equation}
for every $n > 1$.

To any index $\alpha \in A$ we associate the sequence of non-empty finite binary words $(w^\alpha_n)_{n \geq 0}$, where $w_0^\alpha$ can be arbitrarily chosen, while for $n \geq 1$
\begin{equation*}
w^\alpha_n = \left\{\vrule width0pt height12pt\right.\!\!\!
\begin{array}{ll}
v_n & \text{if $\alpha_n = 0$}\\[2pt]
u_n & \text{if $\alpha_n = 1$}
\end{array}.
\end{equation*}
We set $k_n = n_\epsilon(w^\alpha_1) + \ldots + n_\epsilon(w^\alpha_{n-1})$ and observe that $k_n$ is the same for all $\alpha \in A$. 

Then, we apply Lemma \ref{lemmacyl} the sequences $(w^\alpha_n)_{n \geq 0}$ and $(k_n)_{n \geq 1}$ and pick a single binary word $b_\alpha \in \cyl{w^\alpha_0} \cap \C$, out of the uncountable resulting set $B$, such that $\rho^{k_n}(b_\alpha) \in \cyl{w_n^\alpha}$ for every $n \geq 1$. 
Finally, we put $S= \{x_\alpha = \xi(b_\alpha)\}_{\alpha \in A}$. Since all the words $b_\alpha$ and $\rho^k(b_\alpha)$ belong to $\C$, by their construction in the proof of Lemma \ref{lemmacyl}, we have that all the points $x_\alpha$ are distinct and that $R^k(x_\alpha) = \xi(\rho^k(b_\alpha))$ for every $k \geq 1$ and $\alpha \in A$.

To see that $S$ is uniformly DC1-scrambled, take any $\alpha \neq \alpha'$ in $A$. For any $\delta > 0$, let $c(\delta) \geq 1$ be such that
\begin{equation*}
1/2^{2^{c(\delta)/2}} < \delta\,,
\end{equation*}
in such a way that the interval $\xi(\cyl{\rho^h(w^\alpha_n)})$ has width
\begin{equation}\label{cdelta}
\lambda(\xi(\cyl{\rho^h(w^\alpha_n)}) < \delta 
\text{ \ for every $h = 1, \dots, n_\epsilon(w^\alpha_n)-c(\delta) = \textstyle\sum_{k=1}^n 2e_k - c(\delta)$}.
\end{equation}
By construction, there is a sequence $(n_i)_{i \geq 1}$ such that $\alpha_{n_i} = \alpha'_{n_i}$ for every $i \geq 1$. By \eqref{cdelta}, this implies that
\begin{equation}\label{eq1}
\frac{1}{k_{n_i+1}}\,
\#\{k= 1,\dots,k_{n_i+1} \text{ such that } |R^k(x_\alpha) - R^k(x_{\alpha'})| < \delta\} \geq 
\frac{\sum_{k=1}^{n_i} 2e_k - c(\delta)}{k_{n_i+1}}\,.
\end{equation}
Thanks to equation \eqref{expseq}, we have
\begin{equation}\label{eq2}
\textstyle\sum_{k=1}^{n_i} 2e_k = \textstyle\sum_{k=1}^{n_i-1} 2e_k + 2e_{n_i} = 
2(n_i^2 + 1)\displaystyle\frac{e_{n_i}}{n_i^2}\,,
\end{equation}
while
\begin{equation}\label{eq3}
\begin{array}{rcl}
k_{{n_i}+1}
& \!\!\!=\!\!\! &
\textstyle\sum_{k=1}^{n_i}n_\epsilon(w_k^\alpha)
= \textstyle\sum_{k=1}^{n_i}\textstyle\sum_{h=1}^{k}2e_h\\[6pt]
& \!\!\!=\!\!\! &
\textstyle\sum_{k=1}^{n_i-1}\textstyle\sum_{h=1}^{k}2e_h + \sum_{h=1}^{n_i}2e_h\\[6pt]
& \!\!\!\leq\!\!\! & 
2(n_i-1)\textstyle\sum_{k=1}^{n_i-1}e_k + \sum_{k=1}^{n_i}2e_k\\[6pt]
& \!\!\!=\!\!\! &
2(n_i-1)\frac{e_{n_i}}{n_i^2} + \textstyle\sum_{k=1}^{n_i}2e_k\\[4pt]
& \!\!\!=\!\!\! &
2(n_i-1)\frac{e_{n_i}}{n_i^2} + 2(n_i^2 + 1)\displaystyle\frac{e_{n_i}}{n_i^2}\,.
\end{array}
\end{equation}
Therefore, the right hand side of \eqref{eq1} is bounded below by
\begin{equation*}
\frac{n_i^2 + 1}{n_i^2+n_i} - \frac{c(\delta)}{k_{n_i+1}}\,,
\end{equation*} 
which tends to 1 as $n_i$ diverges. This completes the proof of property \eqref{DC11}.

To prove \eqref{DC12}, we consider a different sequence $(n_i)_{i \geq 1}$ such that $\alpha_{n_i} \neq \alpha'_{n_i}$ for every $i \geq 1$, whose existence is once again guaranteed by construction.
Then, \eqref{vnun} and \eqref{manyones} imply that $|R^{k_{n_i}}(x_\alpha) - R^{k_{n_i}}(x_{\alpha'})| > 1/8$. So, putting $\overline\delta = 1/8$ and arguing as above, we get
\begin{equation*}
\frac{1}{k_{n_i+1}}\,
\#\{k= 1,\dots,k_{n_i+1} \text{ such that } |R^k(x_\alpha) - R^k(x_{\alpha'})| < \overline\delta\} \leq 
\frac{k_{n_i+1} - \sum_{k=1}^{n_i} 2e_k}{k_{n_i+1}}\,.
\end{equation*}
Recalling the equations \eqref{eq2} and \eqref{eq3}, we see that the last quantity vanishes as $n_i$ diverges, which gives property \eqref{DC12}.
\end{proof}

We emphasize that in the above proof the family of starting words $(w^\alpha_0)_{\alpha \in A}$ is completely arbitrary. This means that the DC1-scrambled sets for $R$ form a dense subset in the space of all subsets of $[0,1]$ with the Hausdorff pseudo-distance. In particular, the DC1-scrambled set $S$ can be chosen to be a dense subset of any interval $I \subset [0,1]$. Moreover, by suitably modifying the sequences $(w^\alpha_n)_{n \geq 1}$ in the proof, it could be shown that actually $R$ is transitively DC1-chaotic, that is all the points of the uncountable DC1-scrambled set $S$ have dense orbits.

\begin{remark} Proposition \ref{DC1chaos} is interesting in view of the results of {\rm\cite{Steele1}\kern-2pt} and {\rm \cite{Steele2}}, in which the author proves that there is a residual subset in the space of bounded functions of Baire class $1$, and of Baire class $2$ as well, whose elements are neither Devaney nor Li-Yorke chaotic\break $($see {\rm \cite{LY}\kern-2pt} and {\rm\cite{Mai}}$)$. It follows that the map $R$, although generated by a very simple erasing substitution rule, is ``much more chaotic" than the topologically generic map in its Baire class and also in the higher one. 
\end{remark}

In the final part of this section we address topological entropy. 

\medskip

We recall that for a map $f: X \to X$ with $X$ a metric space, the topological entropy can be written as \cite{Dinaburg,Bowen}
\begin{equation}\label{entropydef}
h(f)=\lim_{\varepsilon \to 0} \limsup_{n \to \infty} \frac{\log r(n,\varepsilon)}{n}\,,
\end{equation}
where $r(n,\varepsilon)$ is the maximum cardinality of an $\varepsilon$-separated set (that is a set whose points are at least $\varepsilon$-apart from each other) in the metric
\begin{equation}\label{dn}
d_n(x,y) = \max_{0 \leq i \leq n} |f^i(x) - f^i(y)|\,.
\end{equation}

Although this definition is expressed in metric terms, it turns out to depend only on the underlying topology --- for a modern general reference on topological entropy see \cite{Downarowicz}.

\medskip

For continuous interval maps, positive topological entropy is known to be equivalent distributional chaos, and in particular to DC1-chaos. In our more general context, the positivity of the entropy of the map $R$ should be proved directly. In fact, in the next proposition we prove that $h(R) = \infty$. This can be thought as a consequence of the fact that the substitution $\rho$ is ``efficient enough" in erasing the prefixes, so that all the points ``forget" quickly where they came from. Namely, Proposition \ref{erasing} tells us that the vanishing order of any finite binary word $w$ satisfies the inequality
\begin{equation}\label{fast}
n_\epsilon(w) \leq 2\left \lfloor{\log_2{|w|}}\right \rfloor+2\,.
\end{equation}

\medskip

\begin{proposition}\label{infentropy}
The map $R$ has infinite topological entropy. 
\end{proposition}

\begin{proof}
For every $k,n \geq 1$ and every sequence of words $w_0,w_1,\dots,w_n \in \{0,1\}^k$, we apply Lemma \ref{lemmacyl} to get a single binary word $b_{w_0,w_1,\dots,w_n} \in \cyl{w_00} \cap \C$ such that
for every $i = 0,1,\dots,n$
\begin{equation*}
\rho^{i\ell}(b_{w_0,w_1,\dots,w_n}) \in \cyl{w_i0}\,,
\end{equation*}
where $\ell = 2\lfloor \log_2 |w|\rfloor + 2 \geq n_\epsilon(w_i)$ according to Proposition \ref{erasing}. 
Then, we put
\begin{equation}
S_{k,n} = \{x_{w_0,w_1,\dots,w_n} = \xi(b_{w_0,w_1,\dots,w_n}) \st w_0,w_1,\dots,w_n \in \{0,1\}^k\} \subset [0,1]\,.
\end{equation}
Since all the words $b_{w_0,w_1,\dots,w_n}$ and $\rho^k(b_{w_0,w_1,\dots,w_n})$ belong to $\C$, by their construction in the proof of Lemma \ref{lemmacyl}, we have that all the points $x_{w_0,w_1,\dots,w_n}$ are distinct and that $R^k(x_{w_0,w_1,\dots,w_n}) = \xi(\rho^k(b_{w_0,w_1,\dots,w_n}))$ for every $k \geq 1$ and $w_0,w_1,\dots,w_n \in \{0,1\}^k$.

We note that, for every $k,n \geq 1$ the set $S_{k,n}$ has cardinality $\#S_{k,n}=2^{(n+1)k}$ and it is $\varepsilon$-sepa\-rated in the metric $d_n$ if $\varepsilon < 1/2^{k+1}$. Indeed, given any two different points $x = x_{w_0,w_1,\dots,w_n}$ and $y = x_{w'_0,w'_1,\dots,w'_n}$ in $S_{k,n}$, we have $w_i \neq w_i'$ for some $i=0,\dots,n$, and hence $|R^i(x) - R^i(y)| \geq \varepsilon$, as it can be easily deduced from $x \in \xi(\cyl{w_i0})$ and $y \in \xi(\cyl{w'_i0})$.

Then, based on equation \eqref{entropydef} and putting $k = \lceil-\log_22\varepsilon\rceil -1$ we get
\begin{equation*}
h(R) \geq \lim_{\varepsilon \to 0} \limsup_{n \to \infty} \frac{\log 2^{(n+1)k}}{n}
= \lim_{\varepsilon \to 0} \lim_{n \to \infty} \frac{(n+1)\log 2^{(\lceil-\log_22\varepsilon\rceil -1)}}{n} = \infty\,.
\end{equation*}
\vskip-\lastskip
\vskip-1.4\baselineskip
\end{proof}
\vskip12pt

Recently some attention has been devoted to the study of finer properties of topological entropy, in particular to the characterization of the points around which the entropy concentrates. More precisely, 
one can introduce the notion of relative entropy $h(f,K)$ for a map $f: X \to X$, with $X$ a metric space, and a subset $K \subset X$, as the limit \eqref{entropydef} where $r(n,\epsilon)$ is defined by constraining the $\varepsilon$-separated sets to be contained in $K$. Then, following \cite{Ye}, if $f$ has positive entropy, $x \in X$ is called an entropy point (resp. a full entropy point) if $h(f,K) > 0$ (resp. $h(f,K) = h(f)$) for every closed neighborhood of $x$ in $X$.

\begin{remark}
It is known that, in case of continuous maps with positive entropy on a compact metric space, minimality implies that every point is a full entropy point {\rm \cite{Ye}}. The system $([0,1],R)$ is not minimal (there are periodic points of every period), nevertheless every point is a full entropy point, as an immediate consequence of Proposition \ref{sensi}.
\end{remark}


\section{Arithmetic properties}
\label{arithmetic}

We recall that a Liouville number is a real number $x$ such that, for every $n$, there exists a sequence of rationals $p_i/q_i$ ($p_i,q_i\in \mathbb{Z}, q_i \neq 0$) verifying
\begin{equation}\label{Liouville}
\left|\,\vrule height12pt width0pt
\smash{x - \frac{p_i}{q_i}}\,\right| < \frac{1}{q_i^n}\,.
\end{equation}
\begin{proposition}\label{manyfixtrans}
The map $R$ admits uncountably many fixed points that are Liouville numbers, hence transcendental real numbers.
\end{proposition}

\begin{proof}
For the existence of uncountably many Liouville fixed points, it suffices to observe that, by construction, the fixed point $0.\fixexp{a}$ is a Liouville number as soon as the sequence $a \in A$ increases fast enough (see definition \eqref{fixexp} and Proposition \ref{fixcantor}).
By a classical result \cite{Liouville}, all the fixed points which are Liouville numbers are transcendental. 
\end{proof}

A similar result could be established analogously for periodic points of any given period.
However, this simple argument does not apply to the largest fixed point, for which we have to provide a different proof of transcendence.

\pagebreak 

\begin{proposition}\label{x0trans}
The largest fixed point $x^0$ of the map $R$ is transcendental. 
\end{proposition}

\begin{proof}
Let $p(n,x^0)$ be the cardinality of the set of words of length $n$ appearing in the binary expansion of $x^0$. Then we have
\begin{equation}\label{complexity}
\limsup_{n\to\infty} \frac{p(n,x^0)}{n} < 8\,.
\end{equation}
Indeed, given any $n > 3$, let $k$ be the least positive integer such that $k \geq \log_2 (n/3)$, hence $3 \cdot 2^k \geq n$ and $3 \cdot 2^{k-1} < n$. Then, we split the binary expansion $b^0$ of $x^0$ as the product $b^0 = b_0'b_0''$, where $b_0' = 00101 \prod_{h=0}^{k-1} 1^{3\cdot 2^h} (01)^{3\cdot 2^h}$ and $b_0'' = \prod_{h=k}^\infty 1^{3\cdot 2^h} (01)^{3\cdot 2^h}$. The length of $b_0'$ is
\begin{equation}
5 + 9 \textstyle\sum_{h=0}^{k-1} 2^h = 5 + 9(2^k-1) = 9 \cdot 2^k - 4 < 6 n - 4 < 6n\,.
\end{equation}
Therefore, there are less than $6n$ different subwords of $b^0$ having length $n$ and starting in $b_0'$. We are left to consider the different subwords of length $n$ that are contained in $b_0''$.
Due to the inequality $3 \cdot 2^k \geq n$, each of them consists of a (possibly empty) sequence of 1's followed/preceded by a (possibly empty or truncated on the right/left) sequence of 01's. So, there are at most $2n+3$ such subwords.
Therefore, $p(n,x^0) < 8n + 3$ for every $n \geq 1$, which gives (\ref{complexity}). Thus, by Theorem 3.1 in \cite{Bugeaud2}, $x^0$ is transcendental.
\end{proof}

Along the same lines we can prove the following.

\begin{proposition}
The ``simplest'' periodic point $x^\ell$ of period $2\ell+1$ defined in \eqref{xell} is transcendental for each $\ell \geq 1$.
\end{proposition}

\begin{proof}
We know that the general form for the binary expansion of $x^\ell$ is $b^\ell = w\prod_{h=0}^{\infty}u_h v_h$
where $w,u_h$ and $v_h$ are the words in \eqref{psw}. Proceeding as above, let $k$ be the least positive integer such that $2^{(2\ell +1)k+\ell} \geq n$ and $2^{(2\ell +1)(k-1)+\ell} < n$. 
We then split the binary expansion of $b^\ell$ as $b_{\ell}'b_{\ell}''$, where $b_\ell'=w\prod_{h=0}^{k-1}u_h v_h$ and $b_\ell''=\prod_{h=k}^{\infty}u_h v_h$.
Now, one shows that the length of $b_\ell'$ is bounded above by $12\cdot 2^{2\ell +1} n$, which also gives a bound for the number of different subwords having length $n$ and starting in $b_\ell'$. Moreover, by the same argument used above yields that $b_\ell''$ contributes with at most $2n + 3$ to $p(n,b^\ell)$.
\end{proof}

\section{Some open questions}
\label{problems}

Many natural questions concerning the map $R$ do not have yet an answer. We limit ourselves to list a few of them.

\medskip

Concerning basic properties of the map, we propose the following.

\begin{question}
What is the Hausdorff dimension of the graph $G$ of $R$? 
\end{question}

We note that the self-affine structure of $G$ discussed in Section \ref{graph} seems not to fit into any of the known criteria for the computation of the Hausdorff dimension.

\medskip

From the ergodic point of view, the main question is probably the next one.

\begin{question}
Does $R$ admit an invariant measure which is not purely atomic?
\end{question}

We recall that the singularity of $R$ prevents the existence of absolutely continuous (with respect to Lebesgue) invariant measures.

\medskip

Although the main interest in studying the dynamics of a map such as R comes from a natural generalization of topological dynamics questions to a wider class of functions, like Baire class 1 functions, it could be of interest also in an ergodic context to model specific ``non-equilibrium'' situations in which events that are impossible at one time become possible at another time.

\pagebreak 

\medskip

Concerning points with a dense orbit, heuristics indicate that there is plenty of them, as we have uncountability for two different reasons: in the procedure described in the proof of Lemma \ref{lemmacylbis} there are countably many steps each involving countably many choices; in Proposition \ref{dense_orbit} there are uncountably many possible choices for the sequence $(w_n)_{n \geq 1}$. 

Therefore, it seems reasonable to ask the following.

\begin{question}
Is the set of points with a dense $R$-orbit of second category $($or even co-meagre$)$? Does it have positive $($or even full\/$)$ Lebesgue measure?
\end{question}

About the arithmetic properties of $R$, in the light of the results of Section \ref{arithmetic}, and of the fact that the Thue-Morse number is known to be transcendental, one could ask the question below.

\begin{question}
Is every fixed point of $R$ transcendental? Is every periodic point of $R$ not in $C_0$ or $C_1$ transcendental?
\end{question}

Finally, as said in the Introduction, the map $R$ is the model-case of a more general class of interval maps generated by erasing block substitutions, meaning (in the binary case) substitution rules $s:\{0,1\}^k \to \{0,1\}^*$ such that at least one word $w$ of length $k$ is mapped to $\epsilon$. Looking at these more general maps poses of course new problems.

\begin{question}
What of the properties we established for the map $R$ are true for interval maps generated by which classes of $($binary\/$)$ erasing block substitutions? What new phenomena arise in this extended context?
\end{question}

Let us list some speculation about this last question.

\vskip-7.5pt\vskip0pt\leftmargini25pt
\begin{enumerate}\topsep0pt\itemsep0pt

\item While the map $\rho$ has an ``almost-inverse" map $\sigma$ (up to insertions of the word 00), this is not true for a general erasing substitution rule $s$ as above. Indeed, in the $s$-preimages of a given infinite word $v \in \{0,1\}^\omega$, there could be in principle infinitely many words $\prod_{i=1}^\infty w_i$ with $w_i \in \{0,1\}^k$, such that $s(w_i) \neq \epsilon$ for every $i \geq 1$. This could make the topological structure of the fibers much more complex than in our present case.

\item Even assuming that there is only one word $w_\epsilon \in \{0,1\}^k$ mapped by $s$ to the empty word, if $w_\epsilon \neq 0^k$ we have to consider separately three objects, that is the set $\D$ of dyadic rationals in $[0,1]$ and the two sets
\begin{eqnarray*}
& S_1=\{x\in[0,1] \st x=0.v,\ v=s(w0^\infty),\ w\in\{0,1\}^*\}\,,\\[4pt]
& S_2=\{x\in[0,1] \st x=0.ww_\epsilon^\infty,\ w\in\{0,1\}^{nk}\}\,.
\end{eqnarray*}
It is not difficult to see that these sets coincide in the case when $w_\epsilon = 0^k$, but this is not true in general. Moreover, if $s^{-1}(\epsilon) \subset \{0,1\}^k$ has more than one element, the interval map $f_s:[0,1]\to[0,1]$ generated by the symbolic action of $s$ is in general discontinuous also at some (actually uncountably many) irrationals. 

\item For suitable erasing substitution rules $s$, the map $f_s$ can have stronger properties than $R$. For instance, $s$ can be chosen so that $f_s$ is a Darboux function.

\end{enumerate}
\vskip-6pt

A first study of interval maps generated by general erasing block substitutions is done in \cite{Interval}, but several problems are still open, in particular concerning point 1 above.


\section*{Acknowledgements}

This work is partially supported by the research project PRIN 2017S35EHN$\textunderscore$004 ``Regular and stochastic behaviour in dynamical systems'' of the Italian Ministry of Education and Research.


\section*{List of symbols}

\setlength{\tabcolsep}{0pt}

\begin{tabular}{p{0.175\textwidth}p{0.825\textwidth}}

$\mathbb N$ & natural numbers (including 0) \\
$\mathbb D \subset \mathbb Q$ & dyadic rationals and all rationals, respectively\\[3pt]
$\varphi$ & the golden ratio $(1+\sqrt5)/2$\\[3pt]

$\Q = \mathbb Q \cap [0,1]$ & rationals in $[0,1]$\\
$\D = \mathbb D \cap [0,1]$ & dyadic rationals in $[0,1]$\\
$\N \subset [0,1]$ & normal binary numbers in $[0,1]$\\[3pt]

$\A = \mathbb N^\omega$ & infinite sequences of natural numbers\\
$\B = \{0,1\}^\omega$ & infinite binary sequences/words\\
$\C \subset \B$ & infinite binary sequences/words which are not eventually $0$\\
$\C'\subset \B$ & infinite binary sequences/words which are not eventually $1$\\
$\S \subset \C$ & infinite binary sequences/words with no pair of consecutive 0's\\[3pt]

$\eta:\A \to \C$ & map $(a_1,a_2,\dots) \mapsto 0^{a_1}10^{a_2}1\dots$ 
\hfill \eqref{etadef}\\
$\xi:\B \to [0,1]$ & map $(b_1,b_2,\dots) \mapsto x = 0.b_1b_2\dots$ 
\hfill \eqref{xidef}\\
$\beta: (0,1] \to \C$ & 
map $x \mapsto (b_1,b_2,\dots)$, the unique sequence in $\C$ such that $x = 0.b_1b_2\dots$
\hfill \eqref{zetadef}\\
$\beta': [0,1) \to \C'$ & 
map $x \mapsto (b_1,b_2,\dots)$, the unique sequence in $\C'$ such that $x = 0.b_1b_2\dots$
\hfill \eqref{zetadef}\\[3pt]

$|s|\kern2pt,|w|$ & length of a sequence $s$ or word $w$, respectively\\
$|s|_1\kern2pt,|w|_1$ & number of 1's in a binary sequence or word, respectively\\
$\widetilde s\kern2pt,\widetilde w$ & binary sequence/word obtained by complementing $s/w$, respectively\\[3pt]

$s\trunc{n}\kern2pt,w\trunc{n}$ & $n$-truncation of a sequence $s$ or a word $w$, respectively
\hfill \eqref{trunc}\\
$\cyl{s}\kern2pt,\cyl{w}$ & set of all infinite sequences/words having prefix $s$/$w$, respectively
\hfill \eqref{cyl}\\[3pt]

$\{0,1\}^*$ & possibly empty finite binary sequences/words\\
$\{0,1\}^\infty$ & possibly empty finite or infinite binary sequences/words\\[3pt]

\end{tabular}

\begin{tabular}{p{0.25\textwidth}p{0.75\textwidth}}
$\rho:\{0,1\}^\infty \to \{0,1\}^\infty$ & the main substitution map 
\hfill \eqref{rho}\\
$\tau:\{0,1\}^\infty \to \{0,1\}^\infty$ & an equivalent 2-block substitution map\hfil\break
(defined on even/infinite length binary sequences/words)
\hfill \eqref{tau}\\
$\overline\tau:\{0,1\}^\infty \to \{0,1\}^\infty$ & the substitution map inverse to $\tau$
\hfill \eqref{tauinv}\\
$\sigma:\{0,1\}^\infty \to \{0,1\}^\infty$ & the ``simplest'' substitution map inverse to $\rho$
\hfill \eqref{sigma}\\[3pt]

$R:[0,1] \to [0,1]$ & the real map corresponding to $\rho$ \hfill \eqref{R}\\
$S:[0,1] \to [0,1]$ & the section of $R$ corresponding to $\sigma$ \hfill \eqref{S}\\[3pt]

\end{tabular}

\begin{tabular}{p{0.1\textwidth}p{0.9\textwidth}}
$\epsilon$ & the empty sequence/word\\
$n_\epsilon$ & vanishing order function of binary sequences/words under the action or $\rho$
\hfill \eqref{vanord}\\[3pt]

$\widetilde\rho$ & the substitution map complementary to $\rho$ \hfill \eqref{tilderho}\\
$\rho_k$ & $\rho$ or $\widetilde\rho$, depending on $k$ being an even or odd natural, respectively 
\hfill \eqref{rhok}\\
$\rho_v$ & auxiliary substitution map, coinciding with $\rho_{|v|}$, such that $\rho(vw) = \rho(v)\rho_v(w)$
\hfill \eqref{rhov}\\
$\rho_v^n$ & auxiliary substitution map such that $\rho^n(vw) = \rho^n(v)\rho^n_v(w)$ \hfill \eqref{rhonv}\\[3pt]

$\binexp{a}{b}$ & binary sequence/word obtained from $b$ by inserting $0^{2a_k}$ before each $k$-th 1
\hfill \eqref{binexp}\\
$\fibexp{b}$ & the auxiliary set $\{\binexp{a}{b}\,|\,a \in \A\}$ used to construct the fibers of $\rho$ and $R$
\hfill (\ref{fibereq}\kern2pt,\kern2pt\ref{fibexp})\\
$\fixexp{a}$ & the element of $\Fix\rho$ generated from the ``simples'' one $b^0$ according to $a\in\A$
\hfill \eqref{fixexp}\\[3pt]

$b^0$ & the ``simplest'' fixed point $\fixexp{\per{0}}$ of the map $\rho$
\hfill \eqref{b0}\\[3pt]

$x^0$ & the largest fixed point $\xi(b^0)$ of the map $R$
\hfill \eqref{x0}\\
$x^\ell$ & the ``simplest'' periodic point of the map $R$ having odd order $2\ell+1$
\hfill \eqref{xell}\\
$C_0\kern2pt,C_1$ & the two rational $R$-cycles $\{0,2/3\}$ and $\{1,1/3\}$ respectively\\
$Q_0\kern2pt,Q_1$ & the subsets of $\Q$ attracted by $C_0$ and $C_1$ respectively
\hfill \eqref{Qi}

\end{tabular} 

\vfill\break

\end{document}